\documentclass[12pt]{article}

\usepackage{amssymb,amsmath,amsfonts}
\usepackage{graphicx,color,enumitem}
\usepackage{calrsfs}	
\usepackage{amsthm}	
\usepackage[round]{natbib}

\renewcommand{\d}{{\rm d}}
\newcommand{\e}{{\rm e}}
\renewcommand{\k}{{\bf k}}
\newcommand{\E}{{\mathbb E}}

\renewcommand{\P}{{\mathbb P}}
\newcommand{\Q}{{\mathbb Q}}

\newcommand{\R}{{\mathbb R}}
\renewcommand{\S}{{\mathbb S}}
\newcommand{\N}{{\mathbb N}}

\newcommand{\Ccal}{{\mathcal C}}

\newcommand{\Ecal}{{\mathcal E}}
\newcommand{\Fcal}{{\mathcal F}}
\newcommand{\Gcal}{{\mathcal G}}

\newcommand{\Ical}{{\mathcal I}}

\newcommand{\Pcal}{{\mathcal P}}
\newcommand{\Qcal}{{\mathcal Q}}
\newcommand{\Rcal}{{\mathcal R}}

\newcommand{\Vcal}{{\mathcal V}}

\newcommand{\Xcal}{{\mathcal X}}
\newcommand{\Ycal}{{\mathcal Y}}

\newcommand{\Mid}{{\ \Big|\ }}

\newcommand{\1}[1]{{\boldsymbol 1_{\{#1\}}}}
\newcommand{\oo}{{\bf 1}}

\newcommand{\Rho}{{\rm P}}
\newcommand{\Eta}{{\rm H}}

\newcommand\vvec{\mathaccent"017E }

\newcommand{\Pol}{{\rm Pol}}
\newcommand{\Id}{{\mathrm{Id}}}

\DeclareMathOperator{\tr}{Tr}

\DeclareMathOperator{\Diag}{Diag}

\newtheorem{theorem}{Theorem}

\newtheorem{corollary}[theorem]{Corollary}

\newtheorem{definition}[theorem]{Definition}
\newtheorem{example}[theorem]{Example}

\newtheorem{lemma}[theorem]{Lemma}

\newtheorem{proposition}[theorem]{Proposition}
\newtheorem{remark}[theorem]{Remark}

\numberwithin{equation}{section}
\numberwithin{theorem}{section}

\begin{document}

\title{Polynomial   Diffusions and Applications in Finance\footnote{The authors wish to thank Damien Ackerer, Peter Glynn, Kostas Kardaras, Guillermo Mantilla-Soler, Sergio Pulido, Mykhaylo Shkolnikov, Jordan Stoyanov and Josef Teichmann for useful comments and stimulating discussions. Thanks are also due to the referees, co-editor, and editor for their valuable remarks. The research leading to these results has received funding from the European Research Council under the European Union's Seventh Framework Programme (FP/2007-2013) / ERC Grant Agreement n.~307465-POLYTE.}}
\author{Damir Filipovi\'c\thanks{EPFL and Swiss Finance Institute, Quartier UNIL-Dorigny, Extranef 218, 1015 Lausanne, Switzerland, email: damir.filipovic@epfl.ch} \quad\quad  Martin Larsson\thanks{ETH Zurich, Department of Mathematics, R\"amistrasse 101, CH-8092, Zurich, Switzerland, email: martin.larsson@math.ethz.ch}}
\date{March 13, 2016\\forthcoming in \emph{Finance and Stochastics}}



\maketitle

\begin{abstract}
This paper provides the mathematical foundation for polynomial   diffusions. They play an important role in a growing range of applications in finance, including financial market models for interest rates, credit risk, stochastic volatility, commodities and electricity. Uniqueness of polynomial   diffusions is established via moment determinacy in combination with pathwise uniqueness. Existence boils down to a stochastic invariance problem that we solve for semialgebraic state spaces. Examples include the unit ball, the product of the unit cube and nonnegative orthant, and the unit simplex.
\\[2ex]
\\[2ex]
\\[2ex]
\end{abstract}

\section{Introduction}

This paper provides the mathematical foundation for polynomial   diffusions on a large class of state spaces in $\R^d$. A polynomial   diffusion is characterized by having a linear drift and quadratic diffusion function. In consequence, moments are given in closed form. Such processes represent an extension of the affine class. They play an important role in a growing range of applications in finance, including financial market models of interest rates, credit risk, stochastic volatility, and commodities and electricity.

An arbitrage-free financial market model is determined by a state price density, i.e.~a positive semimartingale $\zeta$ defined on a filtered probability space $(\Omega,\Fcal,\Fcal_t,\P)$. The model price $\Pi(t,T)$ at time $t$ of any time $T$ cash-flow $C_T$ is given by
\begin{equation}\label{eq:PitTNEW}
\Pi(t,T) = \frac{1}{\zeta_t} \E\left[ \zeta_T\, C_T\mid\Fcal_t\right].
\end{equation}
We may interpret $\P$ as the historical measure, or more generally, as an auxiliary measure possibly different from, but equivalent to, the historical measure. A {\em polynomial diffusion model} consists of a polynomial   diffusion $X$ as factor process, along with a positive polynomial $p$ on the state space. The state price density is specified by $\zeta_t = \e^{-\alpha t}p(X_t)$, where $\alpha$ is a real parameter chosen to control the lower bound on implied interest rates. We let the time $T$ cash-flow of a security be given by $C_T=q(X_T)$ for some polynomial~$q$. The polynomial   property of $X$ along with the elementary fact that $p q$ is a polynomial implies that $\Pi(t,T)$ becomes a rational function in $X_t$ with coefficients given in closed form in terms of a matrix exponential. Polynomial diffusion models thus yield closed form expressions for any security with cash-flows  specified as polynomial functions of $X$, which makes them universally applicable in finance. This includes financial market models for interest rates (with $C_T=1$), credit risk in a doubly stochastic framework (with $C_T$ the conditional survival probability), stochastic volatility (with $C_T$ the spot variance), and commodities and electricity (with $C_T$ the spot price).

While polynomial   diffusions have appeared in the literature since \cite{Wong:1964}, so far no existence and uniqueness theory has been available beyond the scalar case. This paper fills this gap and thus provides the mathematical foundation for polynomial diffusion models in finance.

Our main uniqueness result (Theorem~\ref{TNEW:unique}) is based on the classical theory of the moment problem. Since the mixed moments of all finite-dimensional marginal distributions of a polynomial   diffusion are uniquely determined by its generator (Theorem~\ref{T:moments} and Corollary~\ref{C:moments}), uniqueness follows whenever these moments determine the underlying distribution. This is often true, for instance in the affine case or when the state space is compact, or more generally if exponential moments exist; Theorem~\ref{T:expmom} provides sufficient conditions. There are, however, situations where the moment problem approach fails. We therefore provide two additional results based on Yamada--Watanabe type arguments, which give uniqueness in the one-dimensional case (Theorem~\ref{T:unique11}) as well as when the process dynamics exhibits a certain hierarchical structure (Theorem~\ref{T:unique2}). These uniqueness results do not depend on the geometry of the state space.

In order to study existence, we assume that the state space is a basic closed semialgebraic set, i.e.~the nonnegativity set of a finite family of polynomials. Existence reduces to a stochastic invariance problem that we solve under suitable geometric and algebraic conditions on the state space (Theorem~\ref{T:existence}). We also study boundary attainment. In applications it is frequently of interest to know whether the trajectories of a given process may hit the boundary of the state space. In particular, simulating trajectories becomes a much more delicate task if the boundary is attained; see \citet{Lord/Koekkoek/Dijk:2010}. We present sufficient conditions for both attainment and non-attainment that are tight (Theorem~\ref{T:boundary}).

A semialgebraic state space is a natural choice for at least three reasons. First, positive semidefiniteness of the quadratic diffusion matrix boils down to nonnegativity constraints on polynomials. Second, polynomial diffusion models in finance involve polynomials that are required to be positive on the state space. And third, semialgebraic sets turn out to be an ideal setting for employing tools from real algebraic geometry to verify the hypotheses of our existence and boundary attainment results.

We give a detailed analysis of some specific semialgebraic state spaces that do and will play an important role in financial applications, and that illustrate the scope of polynomial   diffusions. Specifically, we consider certain quadric sets including the unit ball $\{x\in\R^d: \|x\|\le 1\}$; the product space $[0,1]^m\times\R^n_+$; and the unit simplex $\{x\in\R^d_+: x_1+\cdots+x_d=1\}$. We also elaborate on polynomial diffusion models in finance, and show how to specify novel stochastic models for interest rates, stochastic volatility, and stock markets.

Polynomial   processes have been studied in various degrees of generality by several authors, for instance \cite{Wong:1964}, \cite{Mazet:1997}, \cite{Zhou2003}, \cite{Forman/Sorensen:2008}, among others. The first systematic accounts treating the time-homogeneous Markov jump-diffusion case are \cite{Cuchiero:2011} and \cite{Cuchiero/etal:2012}. The use of polynomial   diffusions in financial modeling goes back at least to the early 2000s. \cite{Zhou2003} used one-dimensional polynomial   (jump-)diffusions to build short rate models that were estimated to data using a generalized method-of-moments approach, relying crucially on the ability to compute moments efficiently. A short rate model based on the Jacobi process was presented by~\cite{Delbaen/Shirakawa:2002}, and \cite{Larsen/Sorensen:2007} used the same process for exchange rate modeling. The multi-dimensional Jacobi process was studied by~\cite{Gourieroux/Jasiak:2006}, who constructed a stock price model with smooth transitions of drift and volatility regimes. More recently, polynomial   diffusions have featured in the context of financial applications in several papers; see \cite{Filipovic/Gourier/Mancini:2015,Filipovic/Larsson/Trolle:2014} for models of the term structure of variance swap rates and interest rates, respectively, and \citet{Cuchiero/etal:2012} for variance reduction for option pricing and hedging, among other applications. There are several reasons for moving beyond the affine class. In particular, non-trivial dynamics on compact state spaces becomes a possibility, which together with the polynomial   property fits well with polynomial expansion techniques; see also \cite{Filipovic/Mayerhofer/Schneider:2013}. Also on non-compact state spaces one can achieve richer dynamics than in the affine case. Examples of non-affine polynomial   processes include multidimensional Jacobi or Fisher-Wright processes \citep{Ethier:1976,Gourieroux/Jasiak:2006}, Pearson diffusions \citep{Forman/Sorensen:2008}, and Dunkl processes \citep{Dunkl:1992,Gallardo/Yor:2006}.

The rest of the paper is structured as follows. In Section~\ref{S:defPP} we define polynomial   diffusions. Section~\ref{secPEM} is concerned with power and exponential moments. In Section~\ref{secUniqueness} we discuss uniqueness. In Section~\ref{secExBa} we treat existence and boundary attainment.  Section~\ref{S:examples} contains examples of semialgebraic state spaces. Section~\ref{S:applications} outlines various polynomial diffusion models in finance. For the sake of readability most proofs are given in the appendix. Some basic notions from algebraic geometry are reviewed in Section~\ref{A:alg}.

We end this introduction with some notational conventions that will be used throughout this paper. For a function $f:\R^d\to\R$ we write $\{f=0\}$ for the set $\{x\in\R^d:f(x)=0\}$. A {\em polynomial}~$p$ on $\R^d$ is a map $\R^d\to\R$ of the form $\sum_\alpha c_\alpha x_1^{\alpha_1}\cdots x_d^{\alpha_d}$, where the sum runs over all multi-indices $\alpha=(\alpha_1,\ldots,\alpha_d)\in\N^d_0$ and only finitely many of the coefficients $c_\alpha$ are nonzero. Such a representation is unique. The {\em degree} of $p$ is the number $\deg p=\max\{\alpha_1+\cdots+\alpha_d : c_\alpha \ne 0\}$. We let $\Pol(\R^d)$ denote the ring of all polynomials on $\R^d$, and $\Pol_n(\R^d)$ the subspace consisting of polynomials of degree at most~$n$. Let $E$ be a subset of~$\R^d$. A {\em polynomial on~$E$} is the restriction $p=q|_E$ to~$E$ of a polynomial $q\in\Pol(\R^d)$. Its degree is $\deg p=\min\{\deg q : p=q|_E,\, q\in\Pol(\R^d)\}$. We let $\Pol(E)$ denote the ring of polynomials on~$E$,  and $\Pol_n(E)$ the subspace of polynomials on $E$ of degree at most~$n$. Both $\Pol_n(\R^d)$ and $\Pol_n(E)$ are finite-dimensional real vector spaces, but if there are nontrivial polynomials that vanish on~$E$ their dimensions will be different. If $E$ has a nonempty interior then $\Pol_n(\R^d)$ and $\Pol_n(E)$ can be identified. The set of real symmetric $d\times d$ matrices is denoted $\S^d$, and the subset of positive semidefinite matrices is denoted~$\S^d_+$.

\section{Definition of polynomial   diffusions} \label{S:defPP}

Throughout this paper we fix maps $a: \R^d\to\S^d$ and $b:\R^d\to\R^d$ with
\begin{equation} \label{eq:abX}
\text{$a_{ij}\in\Pol_2(\R^d)$ and $b_i\in\Pol_1(\R^d)$ for all $i,j$}
\end{equation}
and a state space $E\subseteq\R^d$. Our goal is to investigate the following issues:
\begin{enumerate}
\item[(a)] For a suitable class of state spaces $E$, find conditions on $a$, $b$, $E$ that guarantee the existence of an $E$-valued solution to the stochastic differential equation
\begin{equation} \label{eq:SDE}
\d X_t = b(X_t)\,\d t + \sigma(X_t)\,\d W_t
\end{equation}
for some $d$-dimensional Brownian motion $W$ and some continuous $\sigma:\R^d\to\R^{d\times d}$ with $\sigma\sigma^\top=a$ on $E$. We will consider the class of basic closed semialgebraic sets $E$, defined using polynomial equalities and inequalities.
\item[(b)] Find conditions for {\em uniqueness in law for $E$-valued solutions to~\eqref{eq:SDE}}. By this we mean that for any $x\in E$ and any $E$-valued solutions $X$ and $X'$ to~\eqref{eq:SDE} with $X_0=X_0'=x$, possibly with different driving Brownian motions, $X$ and $X'$ have the same law.
\item[(c)] Find conditions for a solution to~\eqref{eq:SDE} to attain the boundary of~$E$.
\item[(d)] Find large parametric classes of $a$, $b$, $E$ for which \eqref{eq:SDE} admits a solution.
\end{enumerate}

Investigating these issues is motivated by the fact that diffusions~\eqref{eq:SDE} admit closed form conditional moments and have broad applications in finance, as we shall see below.

We consider the partial differential operator $\Gcal$ given by
\begin{equation} \label{eq:G}
\Gcal f = \frac{1}{2}\tr( a\, \nabla^2 f) + b^\top \nabla f.
\end{equation}
In view of \eqref{eq:abX}, $\Gcal$ maps $\Pol_n(\R^d)$ to itself for each $n\in\N$. As we work on a state space $E\subseteq\R^d$ we now refine this property. We say that $\Gcal$ is {\em well-defined on $\Pol(E)$} if $\Gcal f=0$ on $E$ for any $f\in \Pol(\R^d)$ with $f=0$ on $E$. In this case, $\Gcal$ is well-defined as an operator on~$\Pol(E)$. This always holds if $E$ has a nonempty interior.

\begin{definition} \label{D:PPg}
The operator $\Gcal$ is called {\em polynomial   on $E$} if it is well-defined on $\Pol(E)$, and thus maps $\Pol_n(E)$ to itself for each $n\in\N$. In this case, we call any $E$-valued solution to~\eqref{eq:SDE} a {\em polynomial   diffusion on $E$}.
\end{definition}

It is a simple matter to verify that any second order partial differential operator that maps $\Pol_n(E)$ to itself for each $n\in\N$ is necessarily of the form \eqref{eq:abX} and \eqref{eq:G} on $E$.

\begin{lemma} \label{L:PPchar}
Let $\widetilde\Gcal f = \frac{1}{2}\tr( \widetilde a\, \nabla^2 f) + \widetilde b^\top \nabla f$ be a partial differential operator for some maps $\widetilde a: \R^d\to\S^d$ and $\widetilde b:\R^d\to\R^d$. Assume $\widetilde\Gcal$ is well-defined on $\Pol(E)$. Then the following are equivalent.
\begin{enumerate}
\item\label{L:PPchar1} $\widetilde\Gcal$ maps $\Pol_n(E)$ to itself for each $n\in\N$.
\item\label{L:PPchar2} $\widetilde\Gcal$ maps $\Pol_n(E)$ to itself for $n\in\{1,2\}$.
\item\label{L:PPchar3} The components of $\widetilde a$ and $\widetilde b$ restricted to $E$ lie in $\Pol_2(E)$ and $\Pol_1(E)$, respectively.
\end{enumerate}
In this case, $\widetilde a$ and $\widetilde b$ restricted to $E$ are uniquely determined by the action of $\widetilde\Gcal$ on $\Pol_2(E)$.
\end{lemma}

\begin{proof}
The implications ${\ref{L:PPchar1}}\Rightarrow{\ref{L:PPchar2}}$ and ${\ref{L:PPchar3}}\Rightarrow{\ref{L:PPchar1}}$ are immediate, and the implication ${\ref{L:PPchar2}}\Rightarrow{\ref{L:PPchar3}}$ follows upon applying $\widetilde\Gcal$ to the monomials of degree one and two. In particular, this pins down $\widetilde a$ and $\widetilde b$ on $E$, and thus also establishes the last part of the lemma.  
\end{proof}

In the one-dimensional case, $d=1$, one can classify all polynomial   diffusions on intervals $E$. Indeed, one has $a(x)=a+\alpha x+Ax^2$ and $b(x)=b+\beta x$ for some scalars $a,\alpha,A,b,\beta$, and $E=\{x\in\R:a(x)\ge 0\}$. See \citet{Forman/Sorensen:2008} and \citet{Filipovic/Gourier/Mancini:2015} for details.

The multidimensional case is less trivial. For example, let $d=2$, $E=\R\times\{0\}$, and consider the operator $\Gcal f(x,y)=\frac{1}{2}\partial_{xx}f(x,y)+\partial_y f(x,y)$. This operator is not well-defined on $\Pol(E)$, since the polynomial $f(x,y)=y$ vanishes on $E$, but $\Gcal f(x,y)=1$. On the other hand, $\Gcal$ is the generator of the diffusion $\d X_t=(\d B_t,\d t)$, where $B$ is a one-dimensional Brownian motion. This process immediately leaves~$E$ for any starting point $x\in E$. If, however, an $E$-valued solution to \eqref{eq:SDE} exists for any starting point $x\in E$, then $\Gcal$ is well-defined on $\Pol(E)$. This follows from the following basic positive maximum principle.

\begin{lemma} \label{L:maxprinciple}
Consider $f\in C^2(\R^d)$ and suppose ${\overline x}\in E$ is a maximizer of $f$ over~$E$. If~\eqref{eq:SDE} admits an $E$-valued solution with $X_0=\overline x$, then $\Gcal f({\overline x})\le 0$.
\end{lemma}

\begin{proof}
Let $X$ be an $E$-valued solution to~\eqref{eq:SDE} with~$X_0=\overline x$, and assume for contradiction that $\Gcal f({\overline x})>0$. By definition of global maximizer, $f(x)\le f({\overline x})$ for all $x\in E$. Let $\tau=\inf\{t\ge 0: \Gcal f(X_t)\le 0\}$, and note that $\tau>0$. Then for $t\in(0,\tau)$ we have $f(X_t)\le f({\overline x})$ and $\Gcal f(X_t)>0$, which implies
\[
f(X_{t\wedge\tau}) - f({\overline x}) - \int_0^{t\wedge\tau} \Gcal f(X_s)\d s < 0
\]
for all $t>0$. Thus the left-hand side is a local martingale starting from zero, strictly negative for all $t>0$. This contradiction proves that $\Gcal f({\overline x})\le 0$.  
\end{proof}

Regarding uniqueness, it is crucial to restrict attention to $E$-valued solutions. To illustrate what can otherwise go wrong, consider the stochastic differential equation $\d X_t = -2\sqrt{X_t^-}\d t+ 2\sqrt{X_t^+}\d W_t$, which is well-known to have a unique $\R_+$-valued solution: the zero-dimensional squared Bessel process. However, this stochastic differential equation admits other solutions that do not remain in $\R_+$, for example $X_t = Y_t\1{t\le\tau} - (t-\tau)^2\1{t>\tau}$, where $Y$ is a zero-dimensional squared Bessel process with $Y_0\ge0$ and $\tau=\inf\{t:Y_t=0\}$. Here $\tau$ is finite almost surely.

Note that in Definition~\ref{D:PPg} we require neither uniqueness of solutions to~\eqref{eq:SDE}, nor that~$\Gcal$ be the generator of a Markov process on $E$. There are two reasons for this. First, existence of $E$-valued solutions to~\eqref{eq:SDE} does not in itself imply that those solutions are Markovian. Second, in the context of Markov processes, the polynomial   property holds if and only if the corresponding semigroup leaves $\Pol_n(E)$ invariant for each $n\in\N$. However, this fact, properly phrased, does not require the Markov property. Only It\^o calculus is needed. This observation is crucial for our approach to proving uniqueness. Finally, we remark that a polynomial   diffusion that is also a Markov process is a ``polynomial process'' in the terminology of \cite{Cuchiero/etal:2012}, with vanishing killing rate and no jumps.

\section{Power and exponential moments} \label{secPEM}

Throughout this section we assume that $\Gcal$ is polynomial   on $E$ and let $X$ be an $E$-valued solution to~\eqref{eq:SDE} realized on a filtered probability space $(\Omega,\Fcal,\Fcal_t,\P)$.

For any $n\in\N$, we let $N=N(n,E)$ denote the dimension of ${\rm Pol}_n(E)$. We fix a basis of polynomials $h_1,\dots,h_N$ for ${\rm Pol}_n(E)$ and write
\[
H(x) = (h_1(x),\dots,h_N(x))^\top.
\]
Then for each $p\in\Pol_n(E)$ there exists a unique vector $\vvec p\in\R^{{N}}$ such that
\begin{equation} \label{eq:PT}
p(x)=H(x)^\top \vvec p.
\end{equation}
The restriction of $\Gcal$ to $\Pol_n(E)$ has a unique matrix representation $G\in\R^{{N}\times {N}}$, characterized by the property that $G\,\vvec p$ is the coordinate vector of~$\Gcal p$ whenever $\vvec p$ is the coordinate vector of $p$. That is, we have
\begin{equation} \label{eq:GnPT}
\Gcal p(x) = H(x)^\top G\,\vvec p.
\end{equation}

We now show that $\E[p(X_T) \mid \Fcal_t]$ is indeed well-defined as a polynomial function of~$X_t$. Recall that we do not assume uniqueness of solutions to~\eqref{eq:SDE}, and we do not require~$X$ to be Markov. The proof is given in Section~\ref{secTmoments}.

\begin{theorem} \label{T:moments}
If $\E[\|X_0\|^{2n}]<\infty$, then for any $p\in\Pol_n(E)$ with coordinate representation~$\vvec p\in\R^{{N}}$, we have
\[
\E[p(X_T) \mid \Fcal_t] = H(X_t)^\top \e^{(T-t)G}\,\vvec p, \qquad t\le T.
\]
\end{theorem}

The following result is a direct consequence of Theorem~\ref{T:moments}. Its statement and proof use standard multi-index notation: For a multi-index $\k=(k_1,\ldots,k_d)\in\N^d_0$ we write $|\k |=k_1+\cdots+k_d$ and $x^\k=x_1^{k_1}\cdots x_d^{k_d}$.

\begin{corollary} \label{C:moments}
For any time points $0\le t_1<\cdots<t_m$ and multi-indices $\k(1), \ldots, \k(m)$ such that
\[
\E\left[ \|X_0\|^{2|\k(1)|+\cdots+2|\k(m)|}\right] < \infty,
\]
the expectation $\E[ X_{t_1}^{\k(1)} \cdots X_{t_m}^{\k(m)} ]$ is uniquely determined by $\Gcal$ and the law of $X_0$.
\end{corollary}

\begin{proof}
We prove the result for $m=2$; the general case follows by iteration. Set ${\bf j}=\k(1)$, $\k=\k(2)$, and $n=|{\bf j}|+|\k|$. Since $\E[\|X_0\|^{2|\k|}]<\infty$, Theorem~\ref{T:moments} yields $X_{t_1}^{\bf j}\E[X_{t_2}^{\k}\mid\Fcal_{t_1}]=p(X_{t_1})$ for some polynomial $p\in\Pol_n(E)$ whose coordinate representation $\vvec p$ only depends on $G$. Since $\E[\|X_0\|^{2n}]<\infty$, another application of Theorem~\ref{T:moments} yields
\[
\E[X_{t_1}^{\bf j}X_{t_2}^\k] = \E[\,\E[p(X_{t_1})\mid\Fcal_0]\,] = \E[H(X_0)^\top \e^{t_1G}\,\vvec q \,]
.\]
This proves the corollary. 
\end{proof}

We next provide conditions under which $X_T$ admits finite exponential moments. This result will be used in connection with proving uniqueness in Theorem~\ref{TNEW:unique} below, but is also of interest on its own for applications in finance.\footnote{We thank Mykhaylo Shkolnikov for suggesting a way to improve an earlier version of this result.} Its proof is given in Section~\ref{secTexpmom}.

\begin{theorem} \label{T:expmom}
If
\begin{equation} \label{T:expmom:1}
\E\left[ \e^{\delta \|X_0\|} \right] < \infty \quad\text{for some}\quad \delta>0
\end{equation}
and the diffusion coefficient satisfies the linear growth condition
\begin{equation} \label{eq:alinNEW}
\|a(x)\| \le C(1+\|x\|)  \quad\text{for all $x\in E$}
\end{equation}
for some constant $C$, then for each $t\ge0$ there exists $\varepsilon>0$ with $\E[ \e^{\varepsilon \|X_t\|}] < \infty$.
\end{theorem}

\section{Uniqueness} \label{secUniqueness}

Throughout this section we assume that $\Gcal$ is polynomial   on $E$. We present three results regarding uniqueness in law for $E$-valued solutions to~\eqref{eq:SDE}. Recall that this notion of uniqueness pertains to deterministic initial conditions, as defined under (b) in Section~\ref{S:defPP}.

The first result relies on the fact that the joint moments of all finite-dimensional marginal distributions of a polynomial   diffusion are uniquely determined by $\Gcal$; see Corollary~\ref{C:moments}. Thus uniqueness in law follows if the finite-dimensional marginal distributions are the only ones with these moments. This property is known as {\em determinacy} in the literature on the moment problem, a classical topic in mathematics; references include \cite{Stieltjes:1894,Akhiezer:1965,Berg:1979,Schmudgen:1991,Stoyanov:2000,Kleiber/Stoyanov:2013} and many others.

\begin{lemma} \label{L:unique}
Let $X$ be an $E$-valued solution to~\eqref{eq:SDE}. If for each $t\ge 0$ there exists $\varepsilon>0$ with $\E[\exp(\varepsilon\|X_t\|)]<\infty$, then any $E$-valued solution to~\eqref{eq:SDE} with the same initial law as~$X$ has the same law as $X$. In particular, this holds if \eqref{T:expmom:1} and~\eqref{eq:alinNEW} are satisfied.
\end{lemma}

\begin{proof}
For any $t\ge 0$ and $i\in\{1,\ldots,d\}$, the hypothesis yields $\E[\exp(\varepsilon |X_{i,t}|)]<\infty$ for some $\varepsilon>0$. As a consequence, the moment generating function of $X_{i,t}$ exists and is analytic in $(-\varepsilon,\varepsilon)$, hence equal to its power series expansion, and thus determined by the moments of $X_{i,t}$. By \citet[Theorem~1]{Curtiss:1942}, the moment generating function determines the law of $X_{i,t}$, which thus satisfies the determinacy property. Now, according to \citet[Theorem~3]{Petersen:1982}, determinacy of the one-dimensional marginals of a measure on~$\R^m$ implies determinacy of the measure itself. It follows that determinacy holds for the law of each collection $(X_{t_1},\ldots,X_{t_m})$, $0\le t_1<\cdots<t_m$. By Corollary~\ref{C:moments} the corresponding moments are the same for any $E$-valued solution to~\eqref{eq:SDE} with the same initial law as~$X$. This proves the theorem. 
\end{proof}

If $X_0=x$ is deterministic, then \eqref{T:expmom:1} holds and Lemma~\ref{L:unique} directly yields our first result.

\begin{theorem} \label{TNEW:unique}
If the linear growth condition~\eqref{eq:alinNEW} is satisfied, then uniqueness in law for $E$-valued solutions to~\eqref{eq:SDE} holds.
\end{theorem}

Theorem~\ref{TNEW:unique} assumes the linear growth condition~\eqref{eq:alinNEW} to ensure existence of exponential moments. While valid for all affine diffusions, as well as when~$E$ is compact, this condition excludes some interesting examples, in particular geometric Brownian motion.\footnote{For geometric Brownian motion there is a more fundamental reason to expect that uniqueness cannot be proved via the moment problem: it is well-known that the log-normal distribution is not determined by its moments; see \cite{Heyde:1963}. It thus becomes natural to pose the following question: {\it Can one find a process~$Y$, essentially different from geometric Brownian motion, such that all joint moments of all finite-dimensional marginal distributions,
\[
\E[Y_{t_1}^{\alpha_1} \cdots Y_{t_m}^{\alpha_m}], \qquad m\in\N, \ (\alpha_1,\ldots,\alpha_m)\in\N^m,\ 0\le t_1<\cdots<t_m<\infty,
\]
coincide with those of geometric Brownian motion?} We have not been able to exhibit such a process. Note that any such $Y$ must possess a continuous version. Indeed, the known formulas for the moments of the log-normal distribution imply that for each $T\ge0$, there is a constant $c=c(T)$ such that $\E[(Y_t-Y_s)^4] \le \kappa\,(t-s)^2$ for all $s\le t\le T, \ |t-s|\le 1$, whence Kolmogorov's continuity lemma implies that $Y$ has a continuous version; see \citet[Theorem~I.25.2]{Rogers/Williams:1994}.} Uniqueness for the geometric Brownian motion holds of course, and can be established via the Yamada--Watanabe pathwise uniqueness theorem for one-dimensional diffusions. Our second result records this fact.

\begin{theorem}\label{T:unique11}
If the dimension is $d=1$, then uniqueness in law for $E$-valued solutions to~\eqref{eq:SDE} holds.
\end{theorem}

\begin{proof}
Since $\Gcal$ is polynomial, the drift $b(x)$ in \eqref{eq:SDE} is an affine function on $E$, and the dispersion restricted to $E$ is of the form $\sigma(x) = \sqrt{ \alpha + ax + Ax^2}$ for some real parameters $\alpha, a, A$. Hence $b(x)$ is Lipschitz continuous, and $\sigma(x)$ satisfies
\[ \left(\sigma(x)-\sigma(y)\right)^2 \le \rho_n\left( |x-y|\right),\quad \text{for all $x,y\in E$ with $|x|, |y|\le n$,} \]
where $\rho_n(z)= |a+2nA| z$, for any $n\ge 1$. A localization argument in conjunction with \citet[Theorem~V.40.1]{Rogers/Williams:1994} shows that pathwise unqiueness holds for any $E$-valued solution to~\eqref{eq:SDE}. This in turn implies uniqueness in law; see \citet[Theorem~V.17.1]{Rogers/Williams:1994}. 
\end{proof}

Our third result, in combination with Theorems~\ref{TNEW:unique} and \ref{T:unique11}, yields uniqueness in a wide range of cases that are encountered in applications. The setup is the following. We assume that any $E$-valued solution to~\eqref{eq:SDE} can be partitioned as $X=(Y,Z)$, where $Y$ is an autonomous $m$-dimensional diffusion with closed state space $E_Y\subseteq\R^m$, $Z$ is $n$-dimensional, and $m+n=d$. That is, $(Y,Z)$ solves the stochastic differential equation
\begin{align}
\d Y_t &= b_Y(Y_t)\,\d t \ + \ \sigma_Y(Y_t)\,\d W_t \label{eq:SDE_Y} \\
\d Z_t &= b_Z(Y_t,Z_t)\,\d t \ + \ \sigma_Z(Y_t,Z_t)\,\d W_t, \label{eq:SDE_Z}
\end{align}
for polynomials $b_Y:\R^m\to\R^m$ and $b_Z:\R^m\times\R^n\to\R^n$ of degree one,  continuous maps $\sigma_Y:\R^m\to\R^{m\times d}$ and $\sigma_Z:\R^m\times\R^n\to \R^{n\times d}$, and where $Y$ takes values in $E_Y$. The proof of the following theorem is given in Section~\ref{secTunique2}.

\begin{theorem} \label{T:unique2}
Assume that uniqueness in law for $E_Y$-valued solutions to~\eqref{eq:SDE_Y} holds, and that $\sigma_Z$ is locally Lipschitz in~$z$ locally in~$y$ on~$E$. That is, for each compact subset $K\subseteq E$, there exists a constant~$\kappa$ such that for all $(y,z,y',z')\in K\times K$,
\begin{equation}\label{T:unique2Ass}
   \| \sigma_Z(y,z) - \sigma_Z(y',z') \| \le \kappa \|z-z'\|.
\end{equation}
Then uniqueness in law for $E$-valued solutions to~\eqref{eq:SDE} holds.
\end{theorem}

\section{Existence and boundary attainment} \label{secExBa}

In this section we discuss existence of $E$-valued solutions to~\eqref{eq:SDE}, and give conditions under which the boundary of the state space is attained. The results are stated and proved using some basic concepts from algebra and algebraic geometry. Section~\ref{A:alg} provides a review of the required notions.

Existence of a solution to~\eqref{eq:SDE} with values in $\R^d$ is well known to hold under linear growth conditions; see for instance~\citet[Theorem~IV.2.4]{Ikeda/Watanabe:1981}. The problem at hand thus boils down to finding conditions under which a solution to~\eqref{eq:SDE} takes values in $E$. This is a stochastic invariance problem. In Section~\ref{appL:nonneg} we discuss necessary and sufficient conditions for nonnegativity of certain It\^o processes, which is the basic tool we use for proving stochastic invariance.

We henceforth assume that the state space $E$ is a {\em basic closed semialgebraic set}. Specifically, let $\Pcal$ and $\Qcal$ be finite collections of polynomials on $\R^d$, and define
\begin{equation}\label{eq:E}
E =  \left\{x\in M\, :\ p(x)\ge 0 \text{ for all } p\in\Pcal \right\}
\end{equation}
where
\begin{equation}\label{eq:M}
M = \left\{x\in\R^d\, :\ q(x)=0 \text{ for all } q\in\Qcal\right\}.
\end{equation}
In particular, if $\Qcal=\emptyset$ then $M=\R^d$. The following result provides simple necessary conditions for the invariance of $E$ with respect to~\eqref{eq:SDE}.

\begin{theorem}\label{T:neces}
Suppose there exists an $E$-valued solution to~\eqref{eq:SDE} with $X_0=x$, for any $x\in E$. Then
\begin{enumerate}
\item\label{P:neces_p} $a\nabla p=0$ and $\Gcal p\ge 0$ on $E\cap\{p=0\}$ for each $p\in\Pcal$;
\item\label{P:neces_q} $a\nabla q=0$ and $\Gcal q=0$ on $E$ for each $q\in\Qcal$.
\end{enumerate}
\end{theorem}

\begin{proof}
Pick any $p\in\Pcal$, $x\in E\cap\{p=0\}$, and let $X$ be a solution to~\eqref{eq:SDE} with $X_0=x$. Then $p(X_t)=\int_0^t\Gcal p(X_s)\d s+\int_0^t\nabla p(X_s)^\top \sigma(X_s)\d W_s$ and $p(X)\ge 0$, so {\ref{P:neces_p}} follows by Lemma~\ref{L:nonneg}{\ref{L:nonneg:1}}. To prove {\ref{P:neces_q}} for $q\in\Qcal$, simply apply the same argument to $q$ and~$-q$. 
\end{proof}

The condition $a\nabla p=0$ states, roughly speaking, that at any boundary point of the state space, there can be no diffusive fluctuations orthogonally to the boundary. The condition $\Gcal p\ge 0$ can be interpreted as ``inward-pointing adjusted drift'' at the boundary. The following example shows that it cannot be replaced by a simple ``inward-pointing drift'' condition.

\begin{example}
Consider the bivariate process $(U,V)$ with dynamics
\begin{align*}
\d U_t &= \d W_{1t}						&& U_0\in\R\\
\d V_t &= \alpha \d t + 2\sqrt{V_t}\d W_{2t}		&& V_0\in\R_+,
\end{align*}
where $(W_1,W_2)$ is Brownian motion and $\alpha> 0$. In other words, $U$ is Brownian motion and $V$ is an independent squared Bessel process. The state space is $\R\times\R_+$. Now consider the process $(X,Y)=(U,V-U^2)$. Its dynamics is
\begin{align*}
\d X_t &= \d W_{1t} \\
\d Y_t &= (\alpha-1) \d t - 2X_t \d W_{1t} + 2\sqrt{X_t^2+Y_t}\d W_{2t},
\end{align*}
and its state space is $E=\{(x,y)\in\R^2:x^2+y\ge0\}$, the epigraph of the function $-x^2$. The drift of $(X,Y)$ is $b(x,y)=(0,\alpha-1)$, which points out of the state space at every boundary point, provided $\alpha<1$. Nonetheless, with $p(x,y)=x^2+y$, a calculation yields $\Gcal p(x,y)=\alpha>0$.
\end{example}

As a converse to Theorem~\ref{T:neces}, we now give sufficient conditions for the existence of an $E$-valued solution to~\eqref{eq:SDE}. The proof of the following theorem is given in Section~\ref{secTexistence}.

\begin{theorem}\label{T:existence}
Suppose $E$ satisfies the following geometric and algebraic properties,
\begin{enumerate}[label={\rm(G\arabic*)}]
\item\label{G1} $\nabla r(x)$, $r\in \Qcal$, are linearly independent for all $x\in M$;
\item\label{G2} the ideals generated by $\Qcal\cup\{p\}$ and $M\cap\{p=0\}$ are equal, $(\Qcal\cup\{p\})=\Ical(M\cap\{p=0\})$, for each $p\in\Pcal$;
\end{enumerate}
and the maps $a$ and $b$ satisfy
\begin{enumerate}[label={\rm(A\arabic*)},start=0]
\item\label{T:existence:psd} $a \in\S^d_+$ on $E$;
\item\label{T:existence:p} $a\,\nabla p=0$ on $M\cap\{p=0\}$ and $\Gcal p>0$ on $E\cap\{p=0\}$ for each $p\in\Pcal$;
\item\label{T:existence:q} $a\,\nabla q=0$ and $\Gcal q=0$ on $M$ for each $q\in\Qcal$.
\end{enumerate}
Then $\Gcal$ is polynomial   on $E$, and there exists a continuous map $\sigma:\R^d\to\R^{d\times d}$ with $\sigma\sigma^\top =a$ on $E$ and such that the stochastic differential equation~\eqref{eq:SDE} admits an $E$-valued solution $X$ for any initial law of $X_0$. This solution can be chosen so that it spends zero time in the sets $\{p=0\}$, $p\in\Pcal$. That is,
\begin{equation} \label{T:existence:zt}
\int_0^t \1{p(X_s)=0}\d s = 0 \text{ for all } t\ge0 \text{ and all } p\in\Pcal.
\end{equation}
\end{theorem}

Conditions~\ref{T:existence:p}--\ref{T:existence:q} should be contrasted with the necessary conditions of Theorem~\ref{T:neces}. The latter are somewhat weaker, since they only make statements about $a$ and $b$ on $E$ rather than $M$, and since the inequality in~Theorem~\ref{T:neces}\ref{P:neces_p} is weak. Theorem~\ref{T:existence} can be generalized to allow for weak inequality in \ref{T:existence:p}, at the cost of allowing absorption of the process at the boundary. We do not consider this generalization here.

Condition~\ref{G1} implies that $M$ is an algebraic submanifold in $\R^d$ of dimension $d-|\Qcal|$. The least obvious condition is arguably \ref{G2}. The crucial implication of \ref{G2} is that any polynomial $f$ that vanishes on $M\cap\{p=0\}$ has a representation $f=h\, p$ on $M$ for some polynomial $h$. In conjunction with \ref{T:existence:p} this implies that $a(x)\nabla p(x)$ decays like $p(x)$ as $x\in E$ approaches the boundary set $E\cap\{p=0\}$, for $p\in\Pcal$. This allows one to prove that the local time of $p(X)$ at level zero vanishes, which makes Lemma~\ref{L:nonneg} applicable; see Section~\ref{secTexistence} for the details.

Condition~\ref{G2} is also the least straightforward to verify. We therefore present two sufficient conditions that are easier to check in concrete examples. The first condition is useful when $M=\R^d$, in which case each ideal appearing on the left hand side in \ref{G2} is generated by a single polynomial. This covers many interesting examples, yet yields conditions that are easy to verify in practice. A proof of the following result can be found in \citet[Theorem~4.5.1]{Bochnak/Coste/Roy:1998}.

\begin{lemma} \label{L:irred}
Let $p\in\Pol(\R^d)$ be an irreducible polynomial and $\Vcal(p)$ its zero set. Then $(p)=\Ical(\Vcal(p))$ if and only if $p$ changes sign on $\R^d$, that is, $p(x)p(y)<0$ for some $x,y\in\R^d$.
\end{lemma}

The second condition applies when the ideals generated by the families $\Qcal\cup\{p\}$ with $p\in\Pcal$ are prime and of full dimension.

\begin{lemma} \label{L:prime}
For $p\in\Pcal$, assume that the ideal $(\Qcal\cup\{p\})$ is prime with dimension $d-1-|\Qcal|$, and that there exists some $x\in M\cap\{p=0\}$ such that the vectors $\nabla r(x)$, $r\in\Qcal\cup\{p\}$, are linearly independent. Then $(\Qcal\cup\{p\})=\Ical(M\cap\{p=0\})$.
\end{lemma}

\begin{proof}
This follows directly from \citet[Proposition~3.3.16]{Bochnak/Coste/Roy:1998}. 
\end{proof}

\begin{remark}
Stochastic invariance problems have been studied by a number of authors; see \cite{DaPrato/Frankowska:2004}, \cite{Filipovic/Tappe/Teichmann:2014}, among many others. The approach in these papers is to impose an ``inward-pointing Stratonovich drift'' condition. This breaks down for polynomial   diffusions. Indeed, consider the squared Bessel process
\[
\d X_t = \alpha\, \d t + 2\sqrt{X_t}\,\d W_t,
\]
which is an $\R_+$-valued affine process for $\alpha\ge0$. The stochastic integral cannot always be written in Stratonovich form, since $\sqrt{X}$ fails to be a semimartingale for $0<\alpha<1$. If nonetheless one formally computes the Stratonovich drift, one obtains $\alpha-1$, suggesting that $\alpha\ge 1$ is needed for stochastic invariance of $\R_+$. However, it is well-known that $\alpha\ge 0$ is the correct condition. Our approach is rather in the spirit of \cite{DaPrato/Frankowska:2007} who however focus on stochastic invariance of closed convex sets.
\end{remark}

Apart from existence, Theorem~\ref{T:existence} asserts that $X$ spends zero time in the sets $\{p=0\}$, $p\in\Pcal$, which roughly speaking correspond to boundary segments of the state space. It does not, however, tell us whether these sets are actually hit. The purpose of the following theorem is to give necessary and sufficient conditions for this to occur. The proof is given in Section~\ref{appT:bdry}. The vector $h$ of polynomials appearing in the theorem exists if~\ref{G2} and~\ref{T:existence:p} are satisfied.

\begin{theorem} \label{T:boundary}
Let $X$ be an $E$-valued solution to~\eqref{eq:SDE} satisfying~\eqref{T:existence:zt}. Consider $p\in\Pcal$ and let $h$ be a vector of polynomials such that $a\,\nabla p = h\,p$ on $M$.
\begin{enumerate}
\item\label{T:boundary:1} Assume there exists a neighborhood $U$ of $E\cap\{p=0\}$ such that
\begin{equation} \label{T:bdry:1}
2\,\Gcal p - h^\top\nabla p\ge0 \quad\text{on}\quad E\cap U.
\end{equation}
Then $p(X_t)>0$ for all $t>0$.
\item\label{T:boundary:2} Assume~\ref{G2} holds and
\[
2\,\Gcal p - h^\top\nabla p=0 \quad\text{on}\quad M\cap\{p=0\}.
\]
Then $p(X_t)>0$ for all $t>0$.
\item\label{T:boundary:3} Let $\overline x\in E\cap\{p=0\}$ and assume
\[
\Gcal p(\overline x)\ge 0 \qquad\text{and}\qquad 2\,\Gcal p({\overline x}) - h({\overline x})^\top\nabla p({\overline x})<0.
\]
Then for any $T>0$ there exists $\varepsilon>0$ such that if $\|X_0-\overline x\|<\varepsilon$ almost surely, then $p(X_t)=0$ for some $t\le T$ with positive probability.
\end{enumerate}
\end{theorem}

As a simple example, we may apply Theorem~\ref{T:boundary} to the scalar square-root diffusion $dX_t=(b+\beta X_t)\,dt+\sigma\sqrt{X_t}\,dB_t$ with parameters $b,\sigma> 0$ and $\beta<0$, and where $B$ is a one-dimensional Brownian motion. In this case $E=\R_+$, and $\Pcal$ consists of the single polynomial $p(x)=x$. We have $a(x)p'(x) = \sigma^2 x = \sigma^2 p(x)$, so that $h(x)\equiv \sigma^2$, and thus
\[
2\,\Gcal p(x) - h(x)p'(x) = 2(b+\beta x) - \sigma^2.
\]
It is well known that $X_t>0$ for all $t>0$ if and only if the Feller condition $2b\ge\sigma^2$ holds. Theorem~\ref{T:boundary}\ref{T:boundary:3} gives the necessity of the Feller condition. Theorem~\ref{T:boundary}\ref{T:boundary:1} and \ref{T:boundary:2} together give the sufficiency of the Feller condition. Indeed, suppose that $2b>\sigma^2$ then Theorem~\ref{T:boundary}\ref{T:boundary:1}  applies, while the condition in Theorem~\ref{T:boundary}\ref{T:boundary:2} is not satisfied. Theorem~\ref{T:boundary}\ref{T:boundary:2} in turn applies when $2b=\sigma^2$, while Theorem~\ref{T:boundary}\ref{T:boundary:1} does not.

\section{Examples of semialgebraic state spaces} \label{S:examples}

We now discuss examples of semialgebraic state spaces of interest, where our results are applicable.

\subsection{Some quadric sets}\label{ssqs}

Let $Q\in\S^d$ be nonsingular, and consider the state space $E=\{x\in\R^d:x^\top Qx\le 1\}$. Here $\Pcal$ consists of the single polynomial $p(x)=1-x^\top Qx$, and $M=\R^d$. After a linear change of coordinates we may assume $Q$ is diagonal with $Q_{ii}\in\{+1,-1\}$. We also suppose $Q_{ii}=1$ for at least some $i$, since otherwise $E=\R^d$. State spaces of this type include the closed unit ball, but also non-convex sets like $\{x\in\R^2:x_1^2-x_2^2\le1\}$, whose boundary is a hyperbola. One can also consider complements of such sets; see Remark~\ref{R:compl} below. One interesting aspect of the state spaces investigated here is that they do not admit non-deterministic affine diffusions; this follows directly from Proposition~\ref{P:quad} below, which shows that $a$ is either quadratic or identically zero. This is in contrast to the parabolic state spaces considered by \citet{Spreij/Veerman:2012}.

The following convex cone of polynomial maps plays a key role. Recall that a polynomial $r\in\Pol(\R^d)$ is called {\em homogeneous of degree $k$} if $r(sx)=s^kr(x)$ for all $x\in\R^d$ and~$s>0$.
\[
\Ccal^Q_+ = \left\{c:\R^d\to\S^d_+ :
\begin{array}{l} c_{ij}\in\Pol_2(\R^d) \text{ is homogeneous of degree 2 for all }i,j\\[1mm]
			\text{and } c(x)Qx=0\text{ for all }x\in\R^d
\end{array}\right\}.
\]
Note that the condition $c(x)Qx=0$ is equivalent to $c(x)\nabla p(x)=0$, meaning that all eigenvectors of $c(x)$ with nonzero eigenvalues are orthogonal to $\nabla p(x)$. The proof of the following proposition is given in Section~\ref{secPquad}.

\begin{proposition} \label{P:quad}
Conditions~\ref{G1}--\ref{G2} hold for the state space~$E$. Moreover, the operator~$\Gcal$ satisfies \ref{T:existence:psd}--\ref{T:existence:q} if and only if
\begin{align}
a(x) &= (1-x^\top Qx)\alpha + c(x)  \label{eq:quad:a} \\
b(x) &= \beta + Bx \label{eq:quad:b}
\end{align}
for some $\alpha\in\S^d_+$, $\beta\in\R^d$, $B\in\R^{d\times d}$ and $c\in\Ccal^Q_+$ such that
\begin{equation} \label{eq:quad}
\beta^\top Qx + x^\top B^\top Qx + \frac{1}{2}\tr(c(x)Q)  <  0 \quad\text{for all}\quad x\in\{p=0\}.
\end{equation}
\end{proposition}

\begin{remark}
If $c(x)$ satisfies the linear growth condition $\|c(x)\| \le C(1+\|x\|)$ for all $x\in E$, then $a(x)$ satisfies~\eqref{eq:alinNEW} and uniqueness in law for $E$-valued solutions to~\eqref{eq:SDE} holds by Theorem~\ref{TNEW:unique}. In particular, this holds if $Q$ is positive definite, i.e.~$Q=\Id$, so that $E$ is the unit ball and hence compact.
\end{remark}

\begin{remark} \label{R:compl}
The conditions of Proposition~\ref{P:quad} can easily be modified to cover state spaces of the form $E=\{x\in\R^d:x^\top Qx\ge 1\}$. This amounts to replacing $p$ by $-p$ above, and includes, for example, the complement of the open unit ball. With this modification, Proposition~\ref{P:quad} is still true as stated, except that $-\alpha$ should lie in $\S^d_+$, and the inequality in~\eqref{eq:quad} should be reversed.
\end{remark}

A question that is not addressed by Proposition~\ref{P:quad} is how to describe the set $\Ccal^Q_+$ in more explicit terms. We now provide a class of maps $c\in\Ccal^Q_+$, which yields a large family of polynomial   diffusions on~$E$ that we expect to be useful in applications.

Let $S_k$, $k=1,\ldots,d(d-1)/2$ be a basis for the linear space of skew-symmetric $d\times d$ matrices. Using the skew-symmetry of the $S_k$ together with the fact that $Q^2=\Id$ it is easy to check that any map $c$ of the form
\begin{equation} \label{eq:quad3}
c(x) =  \sum_{k, l=1}^{d(d-1)/2} \gamma_{kl} QS_k xx^\top S_l^\top Q,
\end{equation}
where $\Gamma=(\gamma_{kl})\in\S_+^{d(d-1)/2}$, lies in $\Ccal^Q_+$. For any $c(x)$ of the form~\eqref{eq:quad3}, condition~\eqref{eq:quad} then becomes
\[
\beta^\top Qx + x^\top \left( B^\top Q + \textstyle{\sum_{k,l}} \gamma_{kl} S_k^\top Q S_l \right) x < 0 \quad\text{for all}\quad x\in\{p=0\}.
\]

\subsection{The product space $[0,1]^m\times\R^n_+$}\label{ss01}

Consider the state space $E=[0,1]^m\times\R^n_+$. Here $d=m+n$, and the generating family of polynomials can be taken to be $\Pcal=\{x_i:i=1,\ldots,m+n;\ 1-x_i: i=1,\ldots,m\}$. To simplify notation, introduce index sets $I=\{1,\ldots,m\}$ and $J=\{m+1,\ldots,m+n\}$, and write $x_I$ (resp.~$x_J$) for the subvector of $x\in\R^d$ consisting of the components with indices in $I$ (resp.~$J$). Similarly, for a matrix $A\in\R^{d\times d}$ we write $A_{II}$, $A_{IJ}$, etc.~for the submatrices with indicated row- and column indices. The proof of the following proposition is given in Section~\ref{secP01mn}.

\begin{proposition} \label{P:01mn}
Conditions \ref{G1}--\ref{G2} hold for the state space~$E$. Moreover, the operator~$\Gcal$ satisfies \ref{T:existence:psd}--\ref{T:existence:q} if and only if
\begin{enumerate}
\item\label{P:01mn:a} The matrix $a$ is given by
\begin{align*}
a_{ii}(x)	&= \gamma_i x_i(1-x_i) 			&& (i\in I)\\
a_{ij}(x)	&=0 							&& (i\in I,\ j\in I\cup J,\ i\ne j)\\
a_{jj}(x)	&= \alpha_{jj}x_j^2 + x_j \left(\phi_j + \psi_{(j)}^\top x_I + \pi_{(j)}^\top x_J\right)
											&& (j\in J)\\
a_{ij}(x)	&= \alpha_{ij}x_ix_j 				&& (i,j\in J,\ i\ne j)
\end{align*}
for some $\gamma\in\R^m_+$, $\psi_{(j)}\in\R^m$, $\pi_{(j)}\in\R^n_+$ with $\pi_{(j),j}=0$, $\phi\in\R^n$ with $\phi_j\ge (\psi_{(j)}^-)^\top\oo$, and $\alpha=(\alpha_{ij})_{i,j\in J}\in\S^n$ such that $\alpha+\Diag(\Pi^\top x_J)\Diag(x_J)^{-1}\in\S^n_+$ for all $x_J\in\R^n_{++}$, where $\Pi\in\R^{n\times n}$ is the matrix with columns $\pi_{(j)}$.

\item\label{P:01mn:b} The vector $b$ is given by
\begin{equation} \label{eq:P_01mn:b}
b(x) = \left(
\begin{array}{lllll}
\beta_I	&+& B_{II} x_I				\\
\beta_J	&+& B_{JI}x_I	&+& B_{JJ}x_J	\\
\end{array}
\right)
\end{equation}
for some $\beta\in\R^d$ and $B\in\R^{d\times d}$ such that $(B^-_{i,I\setminus\{i\}})\oo<\beta_i< - B_{ii}-(B^+_{i,I\setminus\{i\}})\oo$ for all $i\in I$, $\beta_j> (B^-_{jI})\oo$ for all $j\in J$, and $B_{JJ}\in\R^{m\times m}$ has positive off-diagonal entries.
\end{enumerate}
\end{proposition}

\begin{remark}
We get uniqueness in the following two cases. First, if $\alpha=0$ and $\pi_{(j)}=0$ for all~$j$, then the linear growth condition~\eqref{eq:alinNEW} is satisfied and uniqueness follows by Theorem~\ref{TNEW:unique}. Second, if $\psi_{(j)}=0$ and $\pi_{(j)}=0$ for all $j$ and $\phi=0$, then the submatrix $a_{JJ}(x)$ only depends on $x_J$ and can be written $a_{JJ}=\sigma_{JJ}\sigma_{JJ}$, where $\sigma_{JJ}(x_J)=\Diag(x_J)\alpha^{1/2}$ is Lipschitz continuous. Since also $X_I$ is an autonomous $m$-dimensional diffusion on $[0,1]^m$, uniqueness follows from Theorem~\ref{T:unique2} in conjunction with Theorem~\ref{TNEW:unique}. Note that $X_I$ and $X_J$ are coupled only through the drift in this case.
\end{remark}

A natural next step is to consider the state space $[0,1]^m\times\R^n_+\times\R^l$, $d=m+n+l$. In this case one readily continues the above argument to deduce that the diffusion matrix is of the form
\[
a(x) = \left(
\begin{array}{lll}
a_{II}(x_I)		& 0					& a_{IK}(x_I)	\\
0			& a_{JJ}(x_I,x_J)		& a_{JK}(x_I,x_J)	\\
a_{IK}(x_I)^\top	& a_{JK}(x_I,x_J)^\top	& a_{KK}(x_I,x_J,x_K)
\end{array}
\right),
\]
where $K=\{m+n+1,\ldots,d\}$, $a_{II}$ and $a_{JJ}$ are given by Proposition~\ref{P:01mn}{\ref{P:01mn:a}}, $a_{IK}(x_I)=\Diag(x_I)(\Id - \Diag(x_I))\Rho$ for some $\Rho\in\R^{m\times l}$, $a_{JK}(x_I,x_J)= \Diag(x_J)\Eta(x_I,x_J)$ for some matrix $\Eta$ of polynomials in $\Pol_1(E)$, and $a_{KK}$ has component functions in $\Pol_2(E)$. Regarding the drift vector $b=(b_I,b_J,b_K)$, the last part $b_K$ is unrestricted within the class of affine functions of $x$, whereas $(b_I,b_J)$ must satisfy Proposition~\ref{P:01mn}{\ref{P:01mn:b}}. With this structure, we have \ref{T:existence:psd}--\ref{T:existence:q} if and only if $a\in\S^d_+$ on $E$. This of course imposes additional restrictions on $\Rho$, $\Eta$, and $a_{KK}$. Stating these restrictions explicitly is cumbersome, and we refrain from doing so here.

\subsection{The unit simplex}\label{ssusx}

Let $d\ge2$ and consider the unit simplex $E=\{x\in\R^d_+: x_1+\cdots+x_d=1\}$. Here $\Pcal=\{x_i:i=1,\ldots,d\}$ consists of the coordinate functions and $\Qcal$ consists of the single polynomial $1-\oo^\top x$. The proof of the following proposition is given in Section~\ref{secPsimplex}.

\begin{proposition} \label{P:simplex}
Conditions \ref{G1}--\ref{G2} hold for the state space~$E$. Moreover, the operator~$\Gcal$ satisfies \ref{T:existence:psd}--\ref{T:existence:q} if and only if
\begin{enumerate}
\item\label{P:simplex:a} The matrix $a$ is given by
\begin{align*}
a_{ii}(x)		&= \sum_{j\ne i}\alpha_{ij}x_ix_j 	&& \\
a_{ij}(x)		&= -\alpha_{ij}x_ix_j 				&& (i\ne j)
\end{align*}
on $E$ for some $\alpha_{ij}\in\R_+$ such that $\alpha_{ij}=\alpha_{ji}$ for all $i,j$.
\item\label{P:simplex:b} The vector $b$ is given by
\[
b(x)=\beta+Bx,
\]
where $\beta\in\R^d$ and $B\in\R^{d\times d}$ satisfy $B^\top\oo + (\beta^\top\oo)\oo = 0$ and $\beta_i+B_{ji} > 0$ for all $i$ and all $j\ne i$.
\end{enumerate}
\end{proposition}

\begin{remark}
Since $E$ is compact, Theorem~\ref{TNEW:unique} yields uniqueness in law for $E$-valued solutions to~\eqref{eq:SDE}.
\end{remark}

\begin{remark}
In the special case where $\alpha_{ij}=\sigma^2$ for some $\sigma>0$ and all $i,j$, the diffusion matrix takes the form
\begin{align*}
a_{ii}(x)		&= \sigma^2x_i(1-x_i) 	&& \\
a_{ij}(x)		&= -\sigma^2x_ix_j		&& (i\ne j).
\end{align*}
The resulting process is sometimes called a {\em multivariate Jacobi} process; see, for instance, \cite{Gourieroux/Jasiak:2006}.
\end{remark}

\begin{remark}
Alternatively, one can establish Proposition~\ref{P:simplex} by considering polynomial   diffusions $Y$ on the ``solid'' simplex $\{y\in\R^{d-1}_+:y_1+\cdots+y_{d-1}\le 1\}$, and then set $X=(X_1,\ldots,X_d)=(Y,1-Y_1-\ldots-Y_{d-1})$. In this case $\Qcal=\emptyset$, and it would be enough to invoke Lemma~\ref{L:irred} rather than Lemma~\ref{L:prime}.
\end{remark}

\section{Polynomial diffusion models in finance} \label{S:applications}

We now elaborate on various polynomial diffusion models in finance, following up on the introduction about \eqref{eq:PitTNEW}. Let the state price density $\zeta$ be a positive semimartingale on a filtered probability space $(\Omega,\Fcal,\Fcal_t,\P)$. This induces an arbitrage-free financial market model on any finite time horizon $T^\ast$. Indeed, let $S^1,\dots,S^m$ denote the price processes of $m$ fundamental assets. According to \eqref{eq:PitTNEW} we have $\zeta_t S^i_{t} =   \E\left[ \zeta_{T^\ast} S^i_{T^\ast} \mid \Fcal_t \right]$. Assuming that $S^1$ is positive, we choose it as numeraire. This implies an equivalent measure $\Q^1\sim\P$ on $\Fcal_{T^\ast}$ by
\[ \frac{\d \Q^1}{\d \P}  =  \frac{\zeta_{T^\ast} {S^1_{T^\ast}}}{\zeta_0 S^1_{0}}.\]
Discounted price processes $\frac{S^i}{ S^1 }$ are $\Q^1$-martingales:
\[ \frac{S^i_t}{ S^1_{t} }\frac{\d \Q^1}{\d \P}|_{\Fcal_t} = \frac{S^i_t}{ S^1_{t}} \frac{\zeta_t  S^1_{t}}{\zeta_0 S^1_{0}}= \frac{\zeta_t S^i_{t}}{\zeta_0 S^1_{0}}. \]
This implies that the market $\{S^1,\dots,S^m\}$ is arbitrage-free in the sense of No Free Lunch with Vanishing Risk, see~\citet{Delbaen/Schachermayer:1994}.

Now let $X$ be a polynomial   diffusion on a state space $E\subseteq\R^d$. Fix $n\in\N$, and let $p\in\Pol_n(E)$ be a positive polynomial on $E$ with coordinate representation $\vvec p$ with respect to some basis $H(x)=(h_1(x),\ldots,h_{N}(x))^\top$ for $\Pol_n(E)$. The state price density is specified by $\zeta_t = \e^{-\alpha t}\,p(X_t)$, where $\alpha$ is a real parameter. This setup yields an arbitrage-free model for the term structure of interest rates. The time-$t$ price $P(t,T)$ of a zero coupon bond maturing at $T$, corresponding to $C_T=1$ in~\eqref{eq:PitTNEW}, can now be computed explicitly using Theorem~\ref{T:moments},
\[
P(t,T) = \e^{-\alpha (T-t)} \frac{\E[p(X_T)\mid\Fcal_t]}{p(X_t)} = \e^{-\alpha (T-t)} \frac{H(X_t)^\top \e^{(T-t)G}\,\vvec p}{H(X_t)^\top \vvec p},
\]
where $G\in\R^{N\times N}$ is the matrix representation of $\Gcal$ on $\Pol_n(E)$. The short rate is obtained via the relation $r_t=-\partial_T\log P(t,T)\mid_{T=t}$, and is given by
\[ r_t = \alpha - \frac{H(X_t)^\top G\,\vvec p }{H(X_t)^\top \vvec p}.\]
This expression clarifies the role of the parameter $\alpha$ adjusting the level of interest rates. Such models show great potential. The linear case with $p$ of the form $p(x)=\phi+\psi^\top x$ has been studied in~\cite{Filipovic/Larsson/Trolle:2014}, including an extensive empirical assessment. The parameter $\psi$ is chosen such that $E$ lies in the positive cone $\{ x\in\R^d: \psi^\top x\ge 0\}$. A specific example is $E=\R^d_+$, as discussed in Section~\ref{ss01}.

One attractive feature of the polynomial framework is that it yields efficient pricing formulae for options on coupon bearing bonds. This includes swaptions, which are among the most important interest rate options. The generic payoff of such an option at expiry date $T$ is of the form
\[
C_{T}=\left(c_0+ c_1 P(T,T_1)+\cdots + c_m P(T,T_m)\right)^+
\]
for maturity dates $T<T_1<\cdots <T_m$ and deterministic coefficients $c_0,\dots,c_m$. Formula \eqref{eq:PitTNEW} for the time $t$ price of this option boils down to computing the $\Fcal_t$-conditional expectation of
\[ \zeta_T \,C_T=\left(H(X_{T})^\top \sum_{i=0}^m c_i \e^{-\alpha T_i} \e^{(T_i-T) G}\,\vvec p\right)^+,\]
which is the positive part of a polynomial in $X_T$. Efficient methods involving the closed form $\Fcal_t$-conditional moments of $X_T$ are available, see \cite{Filipovic/Mayerhofer/Schneider:2013}.

Polynomial   diffusions can be employed in a similar way to build stochastic volatility models. We now interpret $\P$ as risk-neutral measure, and specify the spot variance (squared volatility) of an underlying stock index by $v_t=p(X_t)$. The variance swap rate for period $[t,T]$ is then given in closed form by
\[ {\rm VS}(t,T) = \frac{1}{T-t}\E\left[ \int_t^T v_s \d s\mid\Fcal_t\right] = \frac{1}{T-t} H(X_t)^\top \left( \int_t^T \e^{(s-t)G} \d s\right) \vvec p.
\]
Such models have been successfully employed in~\cite{Filipovic/Gourier/Mancini:2015} and~\cite{Ackerer/Filipovic/Pulido:2015}. Both papers consider the quadratic case, which falls into the setup of Section~\ref{ssqs}, with a quadric state space $E=\{x\in\R^d:x^\top Qx\le 1\}$ and spot variance $v_t=p(X_t)$ for a polynomial $p$ of the form $p(x)=\phi+1-x^\top Q x$ where $\phi\ge 0$ denotes the minimal spot variance. While \cite{Filipovic/Gourier/Mancini:2015} study unbounded state spaces, \cite{Ackerer/Filipovic/Pulido:2015} focus on the compact case, where $Q$ is positive definite. They derive analytic option pricing formula in terms of Hermite polynomials for European call and put options on an asset with diffusive price process $\d S_t  = S_t r \,\d t + S_t \sqrt{v_t}\,\d W^\ast_t$ where $r$ denotes the constant short rate and $W^\ast$ is a Brownian motion, which is possibly correlated with $W$ in \eqref{eq:SDE}.

An application of the unit simplex in Section~\ref{ssusx} is obtained as follows. Consider a stock index, such as the S\&P 500, whose price process is given by a semimartingale $Z$. As above we interpret $\P$ as risk-neutral measure and assume a constant short rate $r$ such that $\e^{-rt}Z_t$ is a martingale. Let $d$ be the number of constituent stocks and let $X$ be a polynomial   diffusion on $E=\{x\in\R^d_+: x_1+\cdots+x_d=1\}$ which is independent of $Z$. We fix a finite time horizon $T^\ast$ and define the $E$-valued martingale, for $t\le T^\ast$,
\[
Y_t = \E[X_{T^\ast}\mid\Fcal_t].
\]
Since $X$ is polynomial, $Y_t$ is a first degree polynomial in $X_t$ whose coefficients can be determined by an application of Theorem~\ref{T:moments}. Specifically, with $\beta$ and $B$ being the drift parameters of $X$ as given in Proposition~\ref{P:simplex}, one finds
\[ \text{$Y_t = \Phi(T^\ast-t) + \Psi(T^\ast-t) X_t $ with $\Phi(\tau)=\int_0^{\tau} \e^{s B}\beta\,\d s$ and $\Psi(\tau)=\e^{\tau B }$.}\]
We now define the constituent stocks' price processes $S^i_t = Y^ i_t Z_t$, $i=1,\dots,d$, such that $S^1+\cdots+S^d=Z$. Assume that the price of the European call option on the index with maturity $T$ and strike $K$ is given in closed form, $C(T,K)$, for some analytic function $C$. The price of the call option on stock $i$ with maturity $T$ and strike $K$ is then given by
\[ C_i(T,K)= \E\left[ Y^i_T \,C(T,K/Y^i_T)\right] .\]
This price can be efficiently computed in three steps. First, compute $\xi \,C(T,K/\xi)$ for a finite set of grid points $\xi\in [0,1]$. Second, apply some polynomial interpolation scheme, for example using Chebyshev polynomials, to obtain a polynomial approximation of degree $n$, say $q(T,K,\xi)$, of $\xi\,C(T,K/\xi)$ in $\xi\in [0,1]$. Third, approximate the option price $C_i(T,K)$ by $H(X_0)^\top \e^{T \,G} \vvec p_i(T,K)$ where $\vvec p_i(T,K)$ is the coordinate representation of the polynomial $p(x)=q\left(T,K,\Phi_i(T^\ast-T) +( \Psi(T^\ast-T) x)_i\right)$ in $x$ with respect to some appropriately chosen basis of polynomials for ${\rm Pol}_n(E)$. Extensions to basket and spread options on the stocks $S^1,\dots,S^d$ are straightforward. This is work in progress.

An application of polynomial   diffusions on a compact state space to credit risk is given in \cite{Ackerer/Filipovic:2015}.

\appendix\normalsize

\section{Nonnegative It\^o processes}\label{appL:nonneg}

The following auxiliary result forms the basis of the proof of Theorem~\ref{T:existence}. It gives necessary and sufficient conditions for nonnegativity of certain It\^o processes.

\begin{lemma} \label{L:nonneg}
Let $Z$ be a continuous semimartingale of the form $Z_t=Z_0+\int_0^t\mu_s\d s+\int_0^t\nu_s\d B_s$, where $Z_0\ge 0$, $\mu$ and $\nu$ are continuous processes, and $B$ is Brownian motion. Let $L^0$ be the local time of $Z$ at level zero.
\begin{enumerate}
\item\label{L:nonneg:2} If $\mu>0$ on $\{Z=0\}$ and $L^0=0$, then $Z\ge 0$ and $\int_0^t \1{Z_s=0}\d s=0$.
\item\label{L:nonneg:1} If $Z\ge 0$, then on $\{Z=0\}$ we have $\mu\ge 0$ and $\nu=0$.
\end{enumerate}
\end{lemma}

\begin{proof}

After stopping we may assume that $Z_t$, $\int_0^t\mu_s\d s$, and $\int_0^t\nu_s\d B_s$ are uniformly bounded. This is done throughout the proof.

We first prove~{\ref{L:nonneg:2}}.  By \citet[Theorem~VI.1.7]{Revuz/Yor:1999} and using that $\mu>0$ on $\{Z=0\}$ and $L^0=0$ we get $0 = L^0_t =L^{0-}_t + 2\int_0^t \1{Z_s=0}\mu_s\d s \ge 0$. In particular, $\int_0^t\1{Z_s=0}\d s=0$, as claimed. Furthermore, Tanaka's formula \cite[Theorem~VI.1.2]{Revuz/Yor:1999} yields
\begin{equation} \label{eq:nonneg:1}
\begin{aligned}
Z_t^- &= -\int_0^t \1{Z_s\le 0}\d Z_s - \frac{1}{2}L^0_t \\
&= -\int_0^t\1{Z_s\le 0}\mu_s \d s - \int_0^t\1{Z_s\le 0}\nu_s \d B_s.
\end{aligned}
\end{equation}
Define stopping times $\rho=\inf\left\{ t\ge 0: Z_t<0\right\}$ and $\tau=\inf\left\{ t\ge \rho: \mu_t=0 \right\} \wedge (\rho+1)$. Using that $Z^-=0$ on $\{\rho=\infty\}$ as well as dominated convergence, we obtain
\[
\E\left[Z^-_{\tau\wedge n}\right] = \E\left[Z^-_{\tau\wedge n}\1{\rho<\infty}\right] \to \E\left[ Z^-_\tau\1{\rho<\infty}\right] \qquad (n\to\infty).
\]
Here $Z_\tau$ is well-defined on $\{\rho<\infty\}$ since $\tau<\infty$ on this set. On the other hand, by \eqref{eq:nonneg:1}, the fact that $\int_0^t\1{Z_s\le0}\mu_s\d s=\int_0^t\1{Z_s=0}\mu_s\d s=0$ on $\{\rho=\infty\}$, and monotone convergence, we get
\begin{align*}
\E\left[Z^-_{\tau\wedge n}\right]
&= \E\left[ - \int_0^{\tau\wedge n}\1{Z_s\le 0}\mu_s\d s\right]  = \E\left[ - \int_0^{\tau\wedge n}\1{Z_s\le 0}\mu_s\d s\, \1{\rho<\infty}\right] \\
&\to  \E\left[ - \int_0^\tau \1{Z_s\le 0}\mu_s\d s\, \1{\rho<\infty}\right] \qquad \text{as $n\to\infty$.}
\end{align*}
Consequently,
\begin{equation}\label{eq:nonneg:2}
\E\left[ Z^-_\tau\1{\rho<\infty}\right] = \E\left[ - \int_0^\tau \1{Z_s\le 0}\mu_s\d s\, \1{\rho<\infty}\right].
\end{equation}
The following hold on $\{\rho<\infty\}$: $\tau>\rho$; $Z_t\ge 0$ on $[0,\rho]$; $\mu_t>0$ on $[\rho,\tau)$; and $Z_t<0$ on some nonempty open subset of $(\rho,\tau)$. Therefore, the random variable inside the expectation on the right-hand side of~\eqref{eq:nonneg:2} is strictly negative on $\{\rho<\infty\}$. The left-hand side, however, is nonnegative, so we deduce $\P(\rho<\infty)=0$. Part~{\ref{L:nonneg:2}} is proved.

The proof of Part~{\ref{L:nonneg:1}} involves the same ideas used, for instance, in \citet[Proposition~3.1]{Spreij/Veerman:2012}. We first assume $Z_0=0$ and prove $\mu_0\ge0$ and $\nu_0=0$. Assume for contradiction that $\P(\mu_0<0)>0$, and define $\tau=\inf\{t\ge0:\mu_t\ge 0\}\wedge 1$. Then $0\le \E[Z_\tau] = \E[\int_0^\tau \mu_s\d s]<0$, a contradiction, whence $\mu_0\ge0$ as desired. Next, pick any $\phi\in\R$ and consider an equivalent measure $\d\Q=\Ecal(-\phi B)_1\d\P$. Then $B^\Q_t = B_t + \phi t$ is $\Q$-Brownian motion on $[0,1]$, and we have
\[
Z_t=\int_0^t(\mu_s-\phi\nu_s)\d s+\int_0^t\nu_s\d B^\Q_s.
\]
Pick any $\varepsilon>0$ and define $\sigma=\inf\{t\ge 0:|\nu_t|\le\varepsilon\}\wedge 1$. The first part of the proof applied to the stopped process $Z^\sigma$ under $\Q$ yields $(\mu_0-\phi\nu_0)\1{\sigma>0}\ge 0$ for all $\phi\in\R$. But this forces $\sigma=0$ and hence $|\nu_0|\le\varepsilon$. Since $\varepsilon>0$ was arbitrary, we get $\nu_0=0$ as desired.

Now, consider any stopping time $\rho$ such that $Z_\rho=0$ on $\{\rho<\infty\}$. Applying what we already proved to the process $(Z_{\rho+t}\1{\rho<\infty})_{t\ge0}$ with filtration $(\Fcal_{\rho+t}\cap\{\rho<\infty\})_{t\ge0}$ then yields $\mu_\rho\ge0$ and $\nu_\rho=0$ on $\{\rho<\infty\}$. Finally, let $\{\rho_n:n\in\N\}$ be a countable collection of such stopping times that are dense in $\{t:Z_t=0\}$. Applying the above result to each $\rho_n$ and using the continuity of $\mu$ and $\nu$, we obtain~{\ref{L:nonneg:1}}. 
\end{proof}

The following two examples show that the assumptions of Lemma~\ref{L:nonneg} are tight in the sense that the gap between {\ref{L:nonneg:2}} and {\ref{L:nonneg:1}} cannot be closed.

\begin{example}
The strict inequality appearing in Lemma~\ref{L:nonneg}{\ref{L:nonneg:2}} cannot be relaxed to a weak inequality: just consider the deterministic process $Z_t=(1-t)^3$.
\end{example}

\begin{example}
The assumption of vanishing local time at zero in Lemma~\ref{L:nonneg}{\ref{L:nonneg:2}} cannot be replaced by the zero volatility condition $\nu=0$ on $\{Z=0\}$, even if the strictly positive drift condition is retained. This is demonstrated by a construction that is closely related to the so-called Girsanov SDE; see \citet[Section~V.26]{Rogers/Williams:1994}. Let $Y$ be a one-dimensional Brownian motion, and define $\rho(y)=|y|^{-2\alpha}\vee 1$ for some $0<\alpha<1/4$. The occupation density formula implies that $\int_0^t\rho(Y_s)^2\d s=\int_{-\infty}^\infty (|y|^{-4\alpha}\vee 1)L^y_t(Y)\d y<\infty$ for all $t\ge 0$, so we may define a positive local martingale
\[
R_t = \exp\left( \int_0^t \rho(Y_s)\d Y_s - \frac{1}{2}\int_0^t \rho(Y_s)^2\d s\right).
\]
Let $\tau$ be a strictly positive stopping time such that the stopped process $R^\tau$ is a uniformly integrable martingale. Then define the equivalent probability measure $\d\Q=R_\tau\d\P$, under which the process $B_t=Y_t-\int_0^{t\wedge\tau}\rho(Y_s)\d s$ is Brownian motion. We now change time via
\[
\varphi_t = \int_0^t \rho(Y_s)\d s, \qquad A_u = \inf\{t\ge 0: \varphi_t > u\},
\]
and define $Z_u = Y_{A_u}$. This process satisfies $Z_u = B_{A_u} + u\wedge\sigma$, where $\sigma=\varphi_\tau$. Define $\beta_u=\int_0^u \rho(Z_v)^{1/2}\d B_{A_v}$, which is Brownian motion since $\langle\beta,\beta\rangle_u=\int_0^u\rho(Z_v)\d A_v=u$. This finally gives
\[
Z_u = \int_0^u (|Z_v|^\alpha\wedge 1) \d \beta_v + u\wedge\sigma.
\]
This process starts at zero, has zero volatility whenever $Z_t=0$, and strictly positive drift prior to the stopping time $\sigma$, which is strictly positive. Nonetheless, its sign changes infinitely often on any time interval $[0,t)$ since it is a time-changed Brownian motion viewed under an equivalent measure.
\end{example}

\section{Proof of Theorem~\ref{T:moments}}\label{secTmoments}

We first establish a lemma.

\begin{lemma} \label{L:exmom}
For any $k\in\N$ such that $\E[\|X_0\|^{2k}]<\infty$, there is a constant $C$ such that
\[
\E\left[ 1 + \|X_t\|^{2k} \mid \Fcal_0\right] \le \left(1+\|X_0\|^{2k}\right)\e^{Ct}, \qquad t\ge 0.
\]
\end{lemma}

\begin{proof}
This is done as in the proof of Theorem~2.10 in~\cite{Cuchiero/etal:2012} using Gronwall's inequality. Specifically, let $f\in\Pol_{2k}(E)$ be given by $f(x)=1+\|x\|^{2k}$, and note that the polynomial   property implies that there is a constant $C$ with $|\Gcal f(x)| \le Cf(x)$ for all $x\in E$. For each $m$, let $\tau_m$ be the first exit time of $X$ from the ball $\{x\in E:\|x\|<m\}$. We can always choose a continuous version of $t\mapsto \E[f(X_{t\wedge\tau_m})\mid\Fcal_0]$, so let us fix such a version. Then by It\^o's formula and the martingale property of $\int_0^{t\wedge\tau_m}\nabla f(X_s)^\top\sigma(X_s)\d W_s$,
\begin{align*}
\E[f(X_{t\wedge\tau_m})\mid\Fcal_0]
&= f(X_0) + \E\left[\int_0^{t\wedge\tau_m}\Gcal f(X_s)\,\d s\Mid \Fcal_0 \right] \\
&\le f(X_0) + C\, \E\left[\int_0^{t\wedge\tau_m} f(X_s)\,\d s\Mid \Fcal_0 \right] \\
&\le f(X_0) + C\int_0^t\E[ f(X_{s\wedge\tau_m})\mid \Fcal_0 ]\, \d s.
\end{align*}
Gronwall's inequality now yields $\E[f(X_{t\wedge\tau_m})\mid\Fcal_0]\le f(X_0)\,\e^{Ct}$. Sending $m$ to infinity and applying Fatou's lemma gives the result. 
\end{proof}

We can now prove Theorem~\ref{T:moments}. For any $p\in\Pol_n(E)$, It\^o's formula yields
\[
p(X_u) = p(X_t) + \int_t^u \Gcal p(X_s) \d s + \int_t^u \nabla p(X_s)^\top \sigma(X_s)\d W_s.
\]
The quadratic variation of the right-hand side satisfies $\int_0^T\nabla p^\top a\, \nabla p(X_s)\d s\le C \int_0^T (1+\|X_s\|^{2n})\d s$ for some constant $C$. This has finite expectation by Lemma~\ref{L:exmom}, so the stochastic integral above is a martingale. Let $\vvec p\in\R^{{N}}$ be the coordinate representation of~$p$. Then~\eqref{eq:PT} and~\eqref{eq:GnPT} in conjunction with the linearity of the expectation and integration operators yield
\begin{align*}
\vvec p^\top \E[H(X_u) \mid \Fcal_t ]
&= \E[p(X_u) \mid \Fcal_t ]  = p(X_t) +  \E[\int_t^u \Gcal p(X_s) \d s\mid\Fcal_t]  \\
&={\vvec p\,}^\top H(X_t) + (G\,\vvec p\,)^\top\E[ \int_t^u  H(X_s)\d s \mid\Fcal_t  ].
\end{align*}
Fubini's theorem, justified by Lemma~\ref{L:exmom}, yields
\[
{\vvec p\,}^\top F(u) = {\vvec p\,}^\top H(X_t) + {\vvec p\,}^\top G^\top \int_t^u  F(s) \d s, \qquad t\le u\le T,
\]
where we define $F(u) = \E[H(X_u) \mid\Fcal_t]$. By choosing unit vectors for $\vvec p$ this gives a system of linear integral equations for $F(u)$, whose unique solution is $F(u)=\e^{(u-t)G^\top}H(X_t)$. Hence
\[
\E[p(X_T) \mid \Fcal_t ] = F(T)^\top \vvec p = H(X_t)^\top \e^{(T-t)G}\,\vvec p,
\]
as claimed. This completes the proof of the theorem.

\section{Proof of Theorem~\ref{T:expmom}}\label{secTexpmom}

Theorem~\ref{T:expmom} is an immediate corollary of the following result.

\begin{lemma}
 Consider the $d$-dimensional It\^o process $X$ with representation
\[ dX_t = (b+\beta X_t)dt + \sigma(X_t)\,dW_t \]
where $\sigma$ satisfies a square-root growth condition
\begin{equation}\label{sigmass}
\text{ $\|\sigma(X_t)\|^2 \le C(1+\|X_t\|)$ for all $t\ge0$}
\end{equation}
for some constant $C$. If
\begin{equation}\label{deltass}
\text{$ \E\left[ \e^{\delta \|X_0\|}\right]<\infty $ for some $\delta>0$,}
\end{equation}
then for each $T\ge 0$ there exists $\varepsilon>0$ with
\begin{equation}\label{Xexpbound}
\E\left[\e^{\varepsilon \|X_T\|}\right]<\infty.
\end{equation}
\end{lemma}

\begin{proof}

Fix $T\ge 0$. Variation of constants lets us rewrite $X_t = A_t + \e^{-\beta(T-t)}Y_t $ with
\[ A_t = \e^{\beta t} X_0+\int_0^t \e^{\beta(t- s)}b\,ds \]
and
\[
Y_t=  \int_0^t \e^{\beta (T- s)}\sigma(X_s)\,dW_s = \int_0^t \sigma^Y_s\,dW_s,
\]
where we write $\sigma^Y_t = \e^{\beta (T- t)}\sigma(A_t + \e^{-\beta(T-t)}Y_t )$. By \eqref{sigmass} the dispersion process $\sigma^Y_t$ satisfies
\begin{equation}\label{Ybound}
 \|\sigma^Y_t\|^2 \le  C_Y(1+\| Y_t\|)
 \end{equation}
for some constant $C_Y$.

Now let $f(y)$ be a real-valued and positive smooth function on $\R^d$ with $f(y)=\sqrt{1+\|y\|}$ for $\|y\|>1$. Some differential calculus gives, for $y\neq 0$,
\[ \nabla \|y\| = \frac{y}{\|y\|} \quad\text{and}\quad \frac{\partial^2 \|y\|}{\partial y_i\partial y_j}=\begin{cases} \frac{1}{\|y\|}-\frac{1}{2}\frac{y_i^2}{\|y\|^3}, & i=j\\ -\frac{1}{2}\frac{y_i y_j}{\|y\|^3} ,& i\neq j .\end{cases}\]
Hence
\[ \nabla f(y)= \frac{1}{2\sqrt{1+\|y\|}}\frac{ y}{\|y\|} \]
and
\[ \frac{\partial^2 f(y)}{\partial y_i\partial y_j}=-\frac{1}{4\sqrt{1+\|y\|}^3}\frac{ y_i}{\|y\|}\frac{ y}{\|y\|}+\frac{1}{2\sqrt{1+\|y\|}}\times
\begin{cases} \frac{1}{\|y\|}-\frac{1}{2}\frac{y_i^2}{\|y\|^3}, & i=j\\ -\frac{1}{2}\frac{y_i y_j}{\|y\|^3} ,& i\neq j \end{cases}\]
for $\|y\|>1$, while first and second order derivate of $f(y)$ are uniformly bounded for $\|y\|\le 1$.

It\^o's formula for $Z_t=f(Y_t)$ gives
\[ dZ_t = \mu^Z_t \,dt +\sigma^Z_t \,dW_t \]
with drift and dispersion processes
\[ \mu^Z_t =  \frac{1}{2}\sum_{i,j=1}^d \frac{\partial^2 f(Y_t)}{\partial y_i\partial y_j} (\sigma^Y_t{\sigma^Y_t}^\top)_{ij} ,\quad \sigma^Z_t= \nabla f(Y_t)^\top\sigma^Y_t .\]
In view of \eqref{Ybound} and the above expressions for $\nabla f(y)$ and $\frac{\partial^2 f(y)}{\partial y_i\partial y_j}$ these are bounded,
\[  \mu^Z_t \le m\quad\text{and}\quad \| \sigma^Z_t \|\le \rho, \]
for some constants $m$ and $\rho$. \cite[Theorem 1.3]{Hajek:1985} now implies that
\[ \E\left[\Phi(Z_T)\right] \le \E\left[\Phi(V)\right] \]
for any nondecreasing convex function $\Phi$ on $\R$, where $V$ is a Gaussian random variable with mean $f(0)+m T$ and variance $\rho^2 T$. Hence, for any $0<\varepsilon' <1/(2\rho^2 T)$ we have $\E[\e^{\varepsilon' V^2}] <\infty$.
We now let $\Phi$ be a nondecreasing convex function on $\R$ with $\Phi(z) = \e^{\varepsilon' z^2}$ for $z\ge 0$. Noting that $Z_T$ is positive, we obtain $\E[ \e^{\varepsilon' Z_T^2}]<\infty$.
As $f(y)^2=1+\|y\|$ for $\|y\|>1$, this implies $\E[ \e^{\varepsilon' \|Y_T\|}]<\infty$.
Combining this with the fact that $\|X_T\| \le \|A_T\| + \|Y_T\| $ and \eqref{deltass}, we obtain using H\"older's inequality the existence of some $\varepsilon>0$ with \eqref{Xexpbound}. 
\end{proof}

\section{Proof of Theorem~\ref{T:unique2}}\label{secTunique2}

We first provide a lemma.

\begin{lemma} \label{L:joint uniqueness}
Assume uniqueness in law holds for $E_Y$-valued solutions to~\eqref{eq:SDE_Y}. Let $Y^1$, $Y^2$ be two $E_Y$-valued solutions to~\eqref{eq:SDE_Y} with driving Brownian motions $W^1$, $W^2$ and with $Y^1_0=Y^2_0=y$ for some $y\in E_Y$. Then $(Y^1,W^1)$ and $(Y^2,W^2)$ have the same law.
\end{lemma}

\begin{proof} Consider the equation $\d Y_t = \widehat b_Y(Y_t)\,\d t + \widehat \sigma_Y(Y_t)\,\d W_t$, where $\widehat b_Y(y)=b_Y(y)\oo_{E_Y}(y)$ and $\widehat \sigma_Y(y)=\sigma_Y(y)\oo_{E_Y}(y)$. Since $E_Y$ is closed, any solution $Y$ to this equation with $Y_0\in E_Y$ must remain inside $E_Y$. To see this, let $\tau=\inf\{t:Y_t\notin E_Y\}$. Then there exists $\varepsilon>0$, depending on $\omega$, such that $Y_t\notin E_Y$ for all $\tau<t<\tau+\varepsilon$. However, since $\widehat b_Y$ and $\widehat \sigma_Y$ vanish outside $E_Y$, $Y_t$ is constant on $(\tau,\tau+\varepsilon)$. Since $E_Y$ is closed this is only possible if $\tau=\infty$.

The hypothesis of the lemma now implies that uniqueness in law for $\R^d$-valued solutions holds for $\d Y_t = \widehat b_Y(Y_t)\,\d t + \widehat \sigma_Y(Y_t)\,\d W_t$. Since $(Y^i,W^i)$, $i=1,2$, are two solutions with $Y^1_0=Y^2_0=y$, \citet[Theorem~3.1]{Cherny:2002} shows that $(W^1,Y^1)$ and $(W^2,Y^2)$ have the same law. 
\end{proof}

The proof of Theorem~\ref{T:unique2} follows along the lines of the proof of the Yamada-Watanabe theorem that pathwise uniqueness implies uniqueness in law; see \citet[Theorem~V.17.1]{Rogers/Williams:1994}. Let $(W^i,Y^i,Z^i)$, $i=1,2$, be $E$-valued weak solutions to \eqref{eq:SDE_Y}--\eqref{eq:SDE_Z} starting from $(y_0,z_0)\in E\subseteq\R^m\times\R^n$. We need to show that $(Y^1,Z^1)$ and $(Y^2,Z^2)$ have the same law. Since uniqueness in law holds for $E_Y$-valued solutions to~\eqref{eq:SDE_Y}, Lemma~\ref{L:joint uniqueness} implies that $(W^1,Y^1)$ and $(W^2,Y^2)$ have the same law, which we denote by $\pi(\d w,\d y)$. Let $Q^i(\d z;w,y)$, $i=1,2$, denote a regular conditional distribution of $Z^i$ given $(W^i,Y^i)$. We equip the path space $C(\R_+,\R^d\times\R^m\times\R^n\times\R^n)$ with the probability measure
\[
\overline\P(\d w,\d y,\d z,\d z') = \pi(\d w, \d y)Q^1(\d z; w,y)Q^2(\d z'; w,y).
\]
Let $(W,Y,Z,Z')$ denote the coordinate process on $C(\R_+,\R^d\times\R^m\times\R^n\times\R^n)$. Then the law under $\overline \P$ of $(W,Y,Z)$ equals the law of $(W^1,Y^1,Z^1)$, and the law under $\overline \P$ of $(W,Y,Z')$ equals the law of $(W^2,Y^2,Z^2)$. By well-known arguments, see for instance \citet[Lemma~V.10.1 and Theorems~V.10.4 and V.17.1]{Rogers/Williams:1994}, it follows that
\begin{align*}
Y_t &= y_0 + \int_0^t b_Y(Y_s)\d s + \int_0^t \sigma_Y(Y_s)\d W_s \\
Z_t &= z_0 + \int_0^t b_Z(Y_s,Z_s)\d s + \int_0^t \sigma_Z(Y_s,Z_s)\d W_s \\
Z'_t &= z_0 + \int_0^t b_Z(Y_s,Z'_s)\d s + \int_0^t \sigma_Z(Y_s,Z'_s)\d W_s.
\end{align*}
By localization we may assume that $b_Z$ and $\sigma_Z$ are Lipschitz in $z$, uniformly in $y$. A standard argument based on the BDG inequalities and Jensen's inequality (see \citet[Corollary~V.11.7]{Rogers/Williams:1994}) together with Gronwall's inequality yields $\overline \P(Z'=Z)=1$. Hence
\[
{\rm Law}(Y^1,Z^1) = {\rm Law}(Y,Z) = {\rm Law}(Y,Z') = {\rm Law}(Y^2,Z^2),
\]
as was to be shown.

\begin{remark}
Theorem~\ref{T:unique2} carries over, and its proof literally goes through, to the case where $(Y,Z)$ is an arbitrary $E$-valued diffusion that solves \eqref{eq:SDE_Y}--\eqref{eq:SDE_Z} and where uniqueness in law for $E_Y$-valued solutions to~\eqref{eq:SDE_Y} holds, provided \eqref{T:unique2Ass} is replaced by the assumption that both $b_Z$ and $\sigma_Z$ are locally Lipschitz in~$z$ locally in~$y$ on $E$. That is, for each compact subset $K\subseteq E$, there exists a constant~$\kappa$ such that for all $(y,z,y',z')\in K\times K$,
\[
\|b_Z(y,z) - b_Z(y',z')\| + \| \sigma_Z(y,z) - \sigma_Z(y',z') \| \le \kappa \|z-z'\|.
\]
\end{remark}

\section{Proof of Theorem~\ref{T:existence}} \label{secTexistence}

The proof of Theorem~\ref{T:existence} consists of two main parts. First, we construct coefficients $\widehat a=\widehat \sigma\widehat \sigma^\top$ and $\widehat b$ that coincide with $a$ and $b$ on $E$, such that a local solution to~\eqref{eq:SDE}, with $b$ and $\sigma$ replaced by $\widehat b$ and $\widehat \sigma$, can be obtained with values in a neighborhood of $E$ in $M$. This relies on~\ref{G1} and \ref{T:existence:q}, and occupies this section up to and including Lemma~\ref{L:existence}. Second, we complete the proof by showing that this solution in fact stays inside~$E$ and spends zero time in the sets $\{p=0\}$, $p\in\Pcal$. This relies on \ref{G2} and \ref{T:existence:p}.

Let $\pi:\S^d\to\S^d_+$ be the Euclidean metric projection onto the positive semidefinite cone. It has the following well-known property.

\begin{lemma} \label{L:proj}
For any symmetric matrix $A\in\S^d$ with spectral decomposition $A=S\Lambda S^\top$, we have $\pi(A)=S\Lambda^+ S^\top$, where $\Lambda^+$ is the element-wise positive part of $\Lambda$.
\end{lemma}

\begin{proof}
This result follows from the fact that the map $\lambda:\S^d\to\R^d$ taking a symmetric matrix to its ordered eigenvalues is 1-Lipschitz; see \citet[Theorem~7.4.51]{Horn/Johnson:1985}. Indeed, for any $B\in\S^d_+$ we have
\[
\|A-S\Lambda^+S^\top\| = \|\lambda(A)-\lambda(A)^+\| \le \|\lambda(A)-\lambda(B)\| \le \|A-B\|.
\]
Here the first inequality uses that the projection of an ordered vector $x\in\R^d$ onto the set of ordered vectors with nonnegative entries is simply $x^+$. 
\end{proof}

We will use the projection $\pi$ to modify the given coefficients $a$ and $b$ outside $E$ in order to obtain candidate coefficients for the stochastic differential equation~\eqref{eq:SDE}. The diffusion coefficients are defined as follows:
\[
\widehat a(x) = \pi\circ a(x), \qquad \widehat \sigma(x) = \widehat a(x)^{1/2}.
\]
In order to construct the drift coefficient $\widehat b$ we need the following lemma.

\begin{lemma} \label{L:driftcorr}
There exists a continuous map $\widehat b :\R^d\to\R^d$ with $\widehat b=b$ on $E$ and such that the operator $\widehat\Gcal$ given by
\[
\widehat\Gcal f = \frac{1}{2}\tr( \widehat a\, \nabla^2 f) + \widehat b\,^{\top} \nabla f
\]
satisfies $\widehat\Gcal f=\Gcal f$ on $E$ and $\widehat \Gcal q = 0 $ on $M$ for all $q\in\Qcal$.
\end{lemma}

\begin{proof}
We first prove that there exists a continuous map $c:\R^d\to\R^d$ such that
\begin{equation}\label{cprop}
\text{$c=0$ on $E$ and $\nabla q^\top c = - \frac{1}{2}\tr\left( (\widehat a-a)\,\nabla^2 q \right) $ on $M$ for all~$q\in\Qcal$. }
\end{equation}
Indeed, let $a=S\Lambda S^\top$ be the spectral decomposition of $a$, so that the columns $S_i$ of $S$ constitute an orthonormal basis of eigenvectors of $a$, and the diagonal elements $\lambda_i$ of $\Lambda$ are the corresponding eigenvalues. These quantities depend on~$x$ in a possibly discontinuous way. For each $q\in\Qcal$,
\begin{equation} \label{eq:driftcorr:1}
\tr\left((\widehat a-a)\,\nabla^2 q \right) = \tr\left( S\Lambda^- S^\top \nabla^2 q\right) = \sum_{i=1}^d \lambda_i^-\, S_i^\top\, \nabla^2q\, S_i.
\end{equation}
Consider now any fixed $x\in M$. For each $i$ such that $\lambda_i(x)^-\ne 0$, $S_i(x)$ lies in the tangent space of~$M$ at~$x$. Thus we may find a smooth path $\gamma_i:(-1,1)\to M$ such that $\gamma_i(0)=x$ and $\gamma_i'(0)=S_i(x)$. For any $q\in\Qcal$ we have $q=0$ on $M$ by definition, whence
\[
0 = \frac{\d^2}{\d s^2} (q \circ \gamma_i)(0) = \tr\left( \nabla^2 q(x)  \gamma_i'(0) \gamma_i'(0)^\top \right) + \nabla q(x)^\top \gamma_i''(0),
\]
or equivalently, $S_i(x)^\top \nabla^2 q(x) S_i(x) = -\nabla q(x)^\top \gamma_i'(0)$. In view of~\eqref{eq:driftcorr:1} this yields
\[
\tr\left((\widehat a(x)- a(x))\,\nabla^2 q(x) \right) = -\nabla q(x)^\top \sum_{i=1}^d \lambda_i(x)^-\gamma_i'(0) \quad\text{for all}\quad q\in\Qcal.
\]
Let $q_1,\ldots,q_m$ be an enumeration of the elements of $\Qcal$, and write the above equation in vector form:
\[
\begin{pmatrix}
\tr\left((\widehat a(x)- a(x))\,\nabla^2 q_1(x) \right) \\
\vdots \\
\tr\left((\widehat a(x)- a(x))\,\nabla^2 q_m(x) \right)
\end{pmatrix}
=
- \begin{pmatrix}
\nabla q_1(x)^\top \\
\vdots \\
\nabla q_m(x)^\top
\end{pmatrix}
\sum_{i=1}^d \lambda_i(x)^-\gamma_i'(0).
\]
The left-hand side thus lies in the range of $[\nabla q_1(x)\ \cdots\ \nabla q_m(x)]^\top$ for each $x\in M$. Since linear independence is an open condition, \ref{G1} implies that the latter matrix has full rank for all $x$ in a whole neighborhood $U$ of $M$. It thus has a Moore-Penrose inverse which is a continuous function of~$x$; see \citet[page~408]{Penrose:1955}. The desired map $c$ is now obtained on $U$ by
\[
c(x) = - \frac{1}{2}\begin{pmatrix}
\nabla q_1(x)^\top \\
\vdots \\
\nabla q_m(x)^\top
\end{pmatrix}^{-1}
\begin{pmatrix}
\tr\left((\widehat a(x)- a(x))\,\nabla^2 q_1(x) \right) \\
\vdots \\
\tr\left((\widehat a(x)- a(x))\,\nabla^2 q_m(x) \right)
\end{pmatrix},
\]
where the Moore-Penrose inverse is understood. Finally, after shrinking $U$ while maintaining $M\subseteq U$, $c$ is continuous on the closure $\overline U$, and can then be extended to a continuous map on $\R^d$ by the Tietze extension theorem; see \citet[Theorem~15.8]{Willard:2004}. This proves \eqref{cprop}.

The extended drift coefficient is now defined by $\widehat b = b + c$, and the operator $\widehat\Gcal$ by
\[
\widehat\Gcal f = \frac{1}{2}\tr( \widehat a\, \nabla^2 f) + \widehat b\,^{\top} \nabla f.
\]
In view of \eqref{cprop} it satisfies $\widehat\Gcal f=\Gcal f$ on $E$ and
\[
 \widehat \Gcal q = \Gcal q + \frac{1}{2}\tr\left( (\widehat a- a)\,\nabla^2 q \right) + c^\top \nabla q = 0
\]
on $M$ for all $q\in\Qcal$, as desired. 
\end{proof}

We now define the set
\[
E_0 = M \cap \{\|\widehat b-b\|<1\}.
\]
Note that $E\subseteq E_0$ since $\widehat b=b$ on $E$. Furthermore, the linear growth condition
\begin{equation} \label{eq:linear growth}
\|\widehat a(x)\|^{1/2} + \|\widehat b(x)\| \le \|a(x)\|^{1/2} + \|b(x)\| + 1 \le C(1+\|x\|),\qquad x\in E_0,
\end{equation}
is satisfied for some constant $C$. This uses that the component functions of $a$ and $b$ lie in $\Pol_2(\R^d)$ and $\Pol_1(\R^d)$, respectively.

An $E_0$-valued local solution to~\eqref{eq:SDE}, with $b$ and $\sigma$ replaced by $\widehat b$ and $\widehat \sigma$, can now be constructed by solving the martingale problem for the operator $\widehat \Gcal$ and state space~$E_0$. We first prove an auxiliary lemma.

\begin{lemma} \label{L:Ghatwd}
Let $f\in C^\infty(\R^d)$ and assume that the support $K$ of $f$ satisfies $K\cap M\subseteq E_0$. Let $x_0$ be a maximizer of $f$ over $E_0$. Then $\widehat \Gcal f(x_0)\le 0$.
\end{lemma}

\begin{proof}
Let $\gamma:(-1,1)\to M$ be any smooth curve in $M$ with $\gamma(0)=x_0$. Optimality of $x_0$ and the chain rule yield
\[
0 = \frac{\d}{\d s} (f \circ \gamma)(0) = \nabla f(x_0)^\top \gamma'(0),
\]
from which it follows that $\nabla f(x_0)$ is orthogonal to the tangent space of $M$ at $x_0$. Thus
\begin{equation} \label{eq:max1}
\nabla f(x_0)=\sum_{q\in\Qcal} c_q \nabla q(x_0)
\end{equation}
for some coefficients $c_q$. Next, differentiating once more yields
\[
0 \ge \frac{\d^2}{\d s^2} (f \circ \gamma)(0) = \tr\left( \nabla^2 f(x_0)  \gamma'(0) \gamma'(0)^\top \right) + \nabla f(x_0)^\top \gamma''(0).
\]
Similarly, for any $q\in\Qcal$,
\[
0 = \frac{\d^2}{\d s^2} (q \circ \gamma)(0) = \tr\left( \nabla^2 q(x_0)  \gamma'(0) \gamma'(0)^\top \right) + \nabla q(x_0)^\top \gamma''(0).
\]
In view of~\eqref{eq:max1}, this implies
\begin{equation} \label{eq:max2}
\tr\left( \Big(\nabla^2 f(x_0) - \sum_{q\in\Qcal} c_q \nabla^2 q(x_0)\Big) \, \gamma'(0) \gamma'(0)^\top \right) \le 0.
\end{equation}
Observe that Lemma~\ref{L:proj} implies that $\ker A\subseteq\ker \pi(A)$ for any symmetric matrix $A$. Thus $\widehat a(x_0)\nabla q(x_0)=0$ for all $q\in\Qcal$ by \ref{T:existence:q}, which implies that $\widehat a(x_0)=\sum_i u_i u_i^\top$ for some vectors $u_i$ in the tangent space of $M$ at $x_0$. Thus, choosing curves $\gamma$ with $\gamma'(0)=u_i$, \eqref{eq:max2} yields
\begin{equation} \label{eq:max3}
\tr\left( \Big(\nabla^2 f(x_0) - \sum_{q\in\Qcal} c_q \nabla^2 q(x_0)\Big) \, \widehat a(x_0) \right) \le 0.
\end{equation}
Combining~\eqref{eq:max1}, \eqref{eq:max3}, and Lemma~\ref{L:driftcorr} we obtain
\[
\widehat\Gcal f(x_0) = \frac{1}{2} \tr\left( \widehat a(x_0)\, \nabla^2 f(x_0) \right) + \widehat b(x_0)^\top \nabla f(x_0)  \le \sum_{q\in\Qcal} c_q\, \widehat\Gcal q(x_0)=0,
\]
as desired. 
\end{proof}

Let $C_0(E_0)$ denote the space of continuous functions on $E_0$ vanishing at infinity. Lemma~\ref{L:Ghatwd} implies that $\widehat \Gcal$ is a well-defined linear operator on $C_0(E_0)$ with domain $C^\infty_c(E_0)$. It also implies that $\widehat \Gcal$ satisfies the positive maximum principle as a linear operator on $C_0(E_0)$. Hence the following local existence result can be proved.

\begin{lemma} \label{L:existence}
Let $\mu$ be a probability measure on $E$. There exists an $\R^d$-valued c\`adl\`ag process $X$ with initial distribution~$\mu$ that satisfies
\begin{equation} \label{eq:SDEhat}
X_t = X_0 + \int_0^t \widehat b(X_s)\,\d s + \int_0^t \widehat \sigma(X_s)\,\d W_s
\end{equation}
for all $t<\tau$, where $\tau = \inf\{t \ge 0: X_t \notin E_0\}>0$, and some $d$-dimensional Brownian motion~$W$.
\end{lemma}

\begin{proof}
The conditions of \citet[Theorem~4.5.4]{Ethier/Kurtz:2005} are satisfied, so there exists an $E_0^\Delta$-valued c\`adl\`ag process $X$ such that $N^f_t = f(X_t) - f(X_0) - \int_0^t \widehat \Gcal f(X_s)\, \d s$ is a martingale for any $f\in C^\infty_c(E_0)$. Here $E_0^\Delta$ denotes the one-point compactification of~$E_0$ with some $\Delta\notin E_0$, and we set $f(\Delta)=\widehat \Gcal f(\Delta)=0$. \citet[Proposition~2]{Bakry/Emery:1985} then yields that $f(X)$ and $N^f$ are continuous.\footnote{Note that, unlike many other results in that paper, Proposition~2 in \citet{Bakry/Emery:1985} does not require $\widehat \Gcal$ to leave $C^\infty_c(E_0)$ invariant, and is thus applicable in our setting.} In particular, $X$ cannot jump to $\Delta$ from any point in $E_0$, whence $\tau$ is a strictly positive predictable time.

A localized version of the argument in \citet[Theorem~5.3.3]{Ethier/Kurtz:2005} now shows that, on an extended probability space, $X$ satisfies~\eqref{eq:SDEhat} for all $t<\tau$ and some Brownian motion~$W$. It remains to show that $X$ is non-explosive in the sense that $\sup_{t<\tau}\|X_\tau\|<\infty$ on $\{\tau<\infty\}$. Indeed, this implies that either $\tau=\infty$, or $\R^d\setminus E_0\neq\emptyset$ in which case we can take $\Delta\in\R^d\setminus E_0$. In either case, $X$ is $\R^d$-valued. To prove that $X$ is non-explosive, let $Z_t=1+\|X_t\|^2$ for $t<\tau$, and observe that the linear growth condition~\eqref{eq:linear growth} in conjunction with It\^o's formula yields $Z_t \le Z_0 + C\int_0^t Z_s\d s + N_t$ for all $t<\tau$, where $C>0$ is a constant and $N$ a local martingale on $[0,\tau)$. Let $Y_t$ denote the right-hand side. Then
\begin{align*}
e^{-tC}Z_t\le e^{-tC}Y_t &= Z_0+C \int_0^t e^{-sC}(Z_s-Y_s)\d s + \int_0^t e^{-sC} \d N_s \\
&\le Z_0 + \int_0^t e^{-s C}\d N_s
\end{align*}
for all $t<\tau$. The right-hand side is a nonnegative supermartingale on $[0,\tau)$, and we deduce $\sup_{t<\tau}Z_t<\infty$ on $\{\tau<\infty\}$, as required.
\end{proof}

Let $X$ and $\tau$ be the process and stopping time provided by Lemma~\ref{L:existence}. We now show that $\tau=\infty$ and that $X_t$ remains in $E$ for all $t\ge0$ and spends zero time in each of the sets $\{p=0\}$, $p\in\Pcal$. This will complete the the proof of Theorem~\ref{T:existence}, since $\widehat a$ and $\widehat b$ coincide with $a$ and $b$ on $E$.

We need to prove that $p(X_t)\ge 0$ for all $0\le t<\tau$ and all $p\in\Pcal$. Fix $p\in\Pcal$ and let $L^y$ denote the local time of $p(X)$ at level~$y$, where we choose a modification that is c\`adl\`ag in~$y$; see \citet[Theorem~VI.1.7]{Revuz/Yor:1999}. It\^o's formula yields
\[
p(X_t) = p(x) + \int_0^t \widehat \Gcal p(X_s) \d s + \int_0^t \nabla p(X_s)^\top \widehat \sigma(X_s)^{1/2}\d W_s, \qquad t<\tau.
\]
We first claim that $L^0_t=0$ for $t<\tau$. The occupation density formula \cite[Corollary~VI.1.6]{Revuz/Yor:1999} yields
\[
\int_{-\infty}^\infty \frac{1}{y}\1{y>0}L^y_t\d y =  \int_0^t \frac{\nabla p^\top \widehat a\,\nabla p(X_s)}{p(X_s)}\1{p(X_s)>0}\d s.
\]
By right continuity of $L^y_t$ in $y$ it suffices to show that the right-hand side is finite. For this, in turn, it is enough to prove that $(\nabla p^\top \widehat a\,\nabla p)/p$ is locally bounded on $M$. To this end, let $a=S\Lambda S^\top$ be the spectral decomposition of $a$, so that the columns $S_i$ of $S$ constitute an orthonormal basis of eigenvectors of $a$, and the diagonal elements $\lambda_i$ of $\Lambda$ are the corresponding eigenvalues. Note that these quantities depend on~$x$ in general. Since $a\,\nabla p=0$ on $M\cap\{p=0\}$ by \ref{T:existence:p}, condition~{\ref{G2}} implies that there exists a vector $h=(h_1,\ldots,h_d)^\top$ of polynomials such that
\[
a\,\nabla p = h\, p \quad \text{on}\quad M.
\]
Thus $\lambda_i S_i^\top\nabla p = S_i^\top a\, \nabla p = S_i^\top h\, p$, and hence $\lambda_i(S_i^\top\nabla p)^2 = S_i^\top\nabla p\, S_i^\top h\, p$. In conjunction with Lemma~\ref{L:proj} this yields
\begin{align*}
\nabla p^\top \widehat a\,\nabla p
&= \nabla p^\top S\Lambda^+ S^\top\nabla p
= \sum_i \lambda_i\1{\lambda_i>0}(S_i^\top\nabla p)^2 \\
&= \sum_i \1{\lambda_i>0}S_i^\top\nabla p\, S_i^\top h\, p.
\end{align*}
Consequently,
\[
\nabla p^\top \widehat a\,\nabla p \le |p| \sum_i \|S_i\|^2 \|\nabla p\| \, \|h\|.
\]
Since $\|S_i\|=1$, and $\nabla p$ and $h$ are locally bounded, we deduce that $(\nabla p^\top \widehat a\,\nabla p)/p$ is locally bounded, as required. Thus $L^0=0$ as claimed.

Next, since $\widehat \Gcal p= \Gcal p$ on $E$, the hypothesis \ref{T:existence:p} implies that $\widehat \Gcal p>0$ on a neighborhood $U_p$ of $E\cap\{p=0\}$. Shrinking $E_0$ if necessary, we may assume that $E_0\subseteq E\cup \bigcup_{p\in\Pcal} U_p$ and thus
\[
\text{$\widehat\Gcal p > 0$ on $E_0\cap\{p=0\}$}.
\]
Since $L^0=0$ before $\tau$, Lemma~\ref{L:nonneg} implies
\[
\text{$p(X_t)\ge 0$ for all $t<\tau$.}
\]
Thus the stopping time $\tau_E=\inf\{t\colon X_t\notin E\}\le\tau$ actually satisfies $\tau_E=\tau$. This implies $\tau=\infty$. Indeed, $X$ has left limits on $\{\tau<\infty\}$ by Lemma~\ref{L:existence}, and $E_0$ is a neighborhood in $M$ of the closed set $E$. Thus $\tau_E<\tau$ on $\{\tau<\infty\}$, whence this set is empty. Finally, Lemma~\ref{L:nonneg} also gives $\int_0^t\1{p(X_s)=0}\d s=0$. The proof of Theorem~\ref{T:existence} is complete.

\section{Proof of Theorem~\ref{T:boundary}} \label{appT:bdry}

The proof of Theorem~\ref{T:boundary} is divided into three parts.

\paragraph{Proof of Theorem~\ref{T:boundary}\ref{T:boundary:1}}
The following argument is a version of what is sometimes called ``McKean's argument''; see \citet[Section~4.1]{Mayerhofer:2011kx} for an overview and further references. Suppose first $p(X_0)>0$ almost surely. It\^o's formula and the identity $a\,\nabla h=h\, p$ on $M$ yield
\begin{equation}\label{eq:log pX}
 \begin{aligned}
\log\, & p(X_t) - \log p(X_0) \\
&= \int_0^t \left(\frac{\Gcal p(X_s)}{p(X_s)} - \frac{1}{2}\frac{\nabla p^\top a\,\nabla p(X_s)}{p(X_s)^2}\right) \d s + \int_0^t \frac{\nabla p^\top \sigma(X_s)}{p(X_s)}\d W_s   \\
&= \int_0^t \frac{2\,\Gcal p(X_s) - h^\top\nabla p(X_s)}{2p(X_s)} \d s + \int_0^t \frac{\nabla p^\top \sigma(X_s)}{p(X_s)}\d W_s
\end{aligned}
\end{equation}
for $t<\tau=\inf\{s\ge0:p(X_s)=0\}$. We will modify $\log p(X)$ to turn it into a local submartingale. To this end, define
\[
V_t = \int_0^t \1{X_s\notin U} \frac{1}{p(X_s)}\left|2\,\Gcal p(X_s) - h^\top\nabla p(X_s)\right| \d s.
\]
We claim that $V_t<\infty$ for all $t\ge 0$. To see this, note that the set $E\cap U^c \cap\{x:\|x\|\le n\}$ is compact and disjoint from $\{p=0\}\cap E$ for each $n$. Thus $\varepsilon_n=\min\{p(x):x\in E\cap U^c,\ \|x\|\le n\}$ is strictly positive. Defining $\sigma_n=\inf\{t:\|X_t\|\ge n\}$, this yields
\[
V_{t\wedge\sigma_n} \le \frac{t}{2\varepsilon_n} \, \max_{\|x\|\le n} \left|2\,\Gcal p(x) - h^\top\nabla p(x)\right| < \infty.
\]
Since $\sigma_n\to\infty$ due to the fact that $X$ does not explode, we have $V_t<\infty$ for all $t\ge0$ as claimed. It follows that the process
\[
A_t = \int_0^t \1{X_s\notin U} \frac{1}{p(X_s)}\left(2\,\Gcal p(X_s) - h^\top\nabla p(X_s)\right) \d s
\]
is well-defined and finite for all $t\ge 0$, with total variation process $V$.

Now, define stopping times $\rho_n=\inf\{t\ge 0: |A_t|+p(X_t) \ge n\}$ and note that $\rho_n\to\infty$ since neither $A$ nor $X$ explodes. Consider the process $Z = \log p(X) - A$, which satisfies
\begin{align*}
Z_t = \log p(X_0) &+ \int_0^t \1{X_s\in U}  \frac{1}{2p(X_s)}\left(2\,\Gcal p(X_s) - h^\top\nabla p(X_s)\right) \d s \\
&+ \int_0^t \frac{\nabla p^\top \sigma(X_s)}{p(X_s)}\d W_s.
\end{align*}
Then $-Z^{\rho_n}$ is a supermartingale on the stochastic interval $[0,\tau)$, bounded from below.\footnote{Details regarding stochastic calculus on stochastic intervals are available in \citet{Maisonneuve:1977}; see also \citet{Mayerhofer:2011kx,Carr/Fisher/Ruf:2014,Larsson/Ruf:2014}.} Thus by the supermartingale convergence theorem, $\lim_{t\uparrow\tau}Z_{t\wedge\rho_n}$ exists in $\R$, which implies $\tau\ge\rho_n$. Since $\rho_n\to\infty$, we deduce $\tau=\infty$, as desired.

Finally, suppose $\P(p(X_0)=0)>0$. The above proof shows that $p(X)$ cannot return to zero once it becomes positive. But due to~\eqref{T:existence:zt} we have $p(X_t)>0$ for arbitrarily small $t>0$, and this completes the proof.

\paragraph{Proof of Theorem~\ref{T:boundary}\ref{T:boundary:2}}

As in the proof of~\ref{T:boundary:1} it suffices to consider the case $p(X_0)>0$. By \ref{G2}  we deduce $2\,\Gcal p - h^\top \nabla p = \alpha p$ on $M$ for some $\alpha\in\Pol(\R^d)$. However, we have $\deg\Gcal p\le\deg p$ and $\deg a\nabla p \le 1+\deg p$, which yields $\deg h\le 1$. Consequently $\deg \alpha p \le \deg p$, implying that $\alpha$ is constant. Inserting this into~\eqref{eq:log pX} yields
\[
\log p(X_t) = \log p(X_0) + \frac{\alpha}{2}t + \int_0^t \frac{\nabla p^\top \sigma(X_s)}{p(X_s)}\d W_s
\]
for $t<\tau=\inf\{t: p(X_t)=0\}$. The process $\log p(X_t)-\alpha t/2$ is thus locally a martingale bounded from above, and hence nonexplosive by the same ``McKean's argument'' as in the proof of part~\ref{T:boundary:1}. This proves the result.

\paragraph{Proof of Theorem~\ref{T:boundary}\ref{T:boundary:3}}

The proof of relies on the following two lemmas.

\begin{lemma} \label{L:exitprob}
Let $b:\R^d\to\R^d$ and $\sigma:\R^d\to\R^{d\times d}$ be continuous functions with $\|b(x)\|^2+\|\sigma(x)\|^2\le\kappa(1+\|x\|^2)$ for some $\kappa>0$, and fix $\rho>0$. Let $Y$ be a $d$-dimensional It\^o process satisfying $Y_t = Y_0 + \int_0^t b(Y_s)\d s + \int_0^t \sigma(Y_s)\d W_s$. Then there exist constants $c_1,c_2>0$ that only depend on $\kappa$ and $\rho$, but not on $Y_0$, such that
\[
\P( \sup_{s\le t}\|Y_s-Y_0\| < \rho )\  \ge \ 1 - t\, c_1\, (1+\E[\|Y_0\|^2]), \qquad t\le c_2.
\]
\end{lemma}

\begin{proof}
By Markov's inequality, $\P( \sup_{t\le\varepsilon}\|Y_t-Y_0\| < \rho )\ge 1-\rho^{-2}\E[\sup_{t\le\varepsilon}\|Y_t-Y_0\|^2]$. Let $\tau_n$ be the first time $\|Y_t\|$ reaches level $n$. A standard argument using the BDG inequality and Jensen's inequality yields
\[
\E\left[ \sup_{s\le t\wedge \tau_n}\|Y_s-Y_0\|^2\right] \le 2c_2\,\E\left[\int_0^{t\wedge\tau_n}\left( \|\sigma(Y_s)\|^2 + \|b(Y_s)\|^2\right)\d s \right]
\]
for $t\le c_2$, where $c_2$ is the constant in the BDG inequality. The growth condition yields
\begin{align*}
\E\left[ \sup_{s\le t\wedge \tau_n}\|Y_s-Y_0\|^2\right]
&\le 2c_2\kappa\,\E\left[\int_0^{t\wedge\tau_n}\left( 1 + \|Y_s\|^2 \right)\d s \right] \\
&\le 4c_2\kappa\,(1+\E[\|Y_0\|^2])t \\
&\qquad + 4c_2\kappa\,\int_0^t\E\left[\sup_{u\le s\wedge\tau_n}\|Y_u-Y_0\|^2 \right]\d s,
\end{align*}
for $t\le c_2$, and Gronwall's lemma then gives $\E[ \sup_{s\le t\wedge \tau_n}\|Y_s-Y_0\|^2] \le c_3t\,\e^{4c_2\kappa t}$, where $c_3=4c_2\kappa\,(1+\E[\|Y_0\|^2])$. Sending $n$ to infinity and applying Fatou's lemma concludes the proof, upon setting $c_1=4c_2\kappa\,\e^{4c_2^2\kappa}\wedge c_2$. 
\end{proof}

\begin{lemma} \label{L:Bessel}
Let $0<\alpha<2$ and $z\ge0$, and let $Z$ be a ${\rm BESQ}(\alpha)$ process starting from $z\ge 0$. Let $\P_z$ denoting its law. Let $\tau_0=\inf\{t\ge0:Z_t=0\}$ be the first time $Z$ hits zero. Then, for any $\varepsilon>0$,
\[
\lim_{z\to0}\P_z(\tau_0>\varepsilon) = 0.
\]
\end{lemma}

\begin{proof}
By \citet[Eq.~(15)]{Going-Jaeschke/Yor:2003}, we have
\[
\P_z(\tau_0>\varepsilon) = \int_\varepsilon^\infty \frac{1}{t\Gamma(\widehat\nu)}\left(\frac{z}{2t}\right)^{\widehat\nu} \e^{-z/(2t)}\d t,
\]
where $\Gamma(\cdot)$ is the Gamma function and $\widehat\nu=1-\alpha/2\in(0,1)$. Changing variable to $s=z/(2t)$ yields $\P_z(\tau_0>\varepsilon)=\frac{1}{\Gamma(\widehat\nu)}\int_0^{z/(2\varepsilon)}s^{\widehat\nu-1}\e^{-s}\d s$, which converges to zero as $z\to0$ by dominated convergence. 
\end{proof}

We may now complete the proof of Theorem~\ref{T:boundary}\ref{T:boundary:3}. The hypotheses yield
\[
0 \le 2\,\Gcal p({\overline x}) < h({\overline x})^\top\nabla p({\overline x}).
\]
Hence there exist some $\delta>0$ such that $2\,\Gcal p({\overline x}) < (1-2\delta) h({\overline x})^\top\nabla p({\overline x})$, and an open ball $U$ in $\R^d$ of radius $\rho>0$, centered at ${\overline x}$, such that
\[
2\,\Gcal p \le \left(1-\delta\right) h^\top\nabla p
\quad\text{and}\quad
h^\top\nabla p >0 \qquad \text{on}\qquad E\cap U.
\]
Note that the radius $\rho$ does not depend on the starting point $X_0$.

For all $t<\tau(U)=\inf\{s\ge0:X_s\notin U\}\wedge T$, we have
\begin{align*}
p(X_t) - p(X_0) - \int_0^t\Gcal p(X_s)\d s
&= \int_0^t \nabla p^\top \sigma(X_s)\d W_s \\
&= \int_0^t \sqrt{\nabla p^\top a\nabla p(X_s)}\d B_s\\
&= 2\int_0^t \sqrt{p(X_s)}\, \frac{1}{2}\sqrt{h^\top \nabla p(X_s)}\d B_s
\end{align*}
for some one-dimensional Brownian motion, possibly defined on an enlargement of the original probability space. Here the equality $a\nabla p =hp$ on $E$ was used in the last step. Define an increasing process $A_t=\int_0^t\frac{1}{4}h^\top\nabla p(X_s)\d s$. Since $h^\top\nabla p(X_t)>0$ on $[0,\tau(U))$, the process $A$ is strictly increasing there. It follows that the time-change $\gamma_u=\inf\{t\ge0:A_t>u\}$ is continuous and strictly increasing on $[0,A_{\tau(U)})$. The time-changed process $Y_u=p(X_{\gamma_u})$ thus satisfies
\[
Y_u = p(X_0) + \int_0^u \frac{4\,\Gcal p(X_{\gamma_v})}{h^\top\nabla p(X_{\gamma_v})}\d v + 2\int_0^u \sqrt{Y_v}\d\beta_v, \qquad u< A_{\tau(U)}.
\]
Consider now the ${\rm BESQ}(2-2\delta)$ process $Z$ defined as the unique strong solution to the equation
\[
Z_u = p(X_0) + (2-2\delta)u + 2\int_0^u \sqrt{Z_v}\d\beta_v.
\]
Since $4\,\Gcal p(X_t) / h^\top\nabla p(X_t) \le 2-2\delta$ for $t<\tau(U)$, a standard comparison theorem implies that $Y_u\le Z_u$ for $u< A_{\tau(U)}$; see for instance \citet[Theorem~V.43.1]{Rogers/Williams:1994}. It is well-known that a BESQ$(\alpha)$ process hits zero if and only if $\alpha<2$; see \citet[page~442]{Revuz/Yor:1999}. It thus remains to exhibit $\varepsilon>0$ such that if $\|X_0-\overline x\|<\varepsilon$ almost surely, there is a positive probability that $Z_u$ hits zero before $X_{\gamma_u}$ leaves $U$, or equivalently, that $Z_u=0$ for some $u<A_{\tau(U)}$. To this end, set $C=\sup_{x\in U} h(x)^\top\nabla p(x)/4$, so that $A_{\tau(U)}\ge C\tau(U)$, and let $\eta>0$ be a number to be determined later. We have
\begin{equation}\label{eq:00111}
 \begin{aligned}
\P\big( \eta < A_{\tau(U)} &\text{ and } \inf_{u\le \eta} Z_u = 0\big) \\
&\ge \P\big( \eta < A_{\tau(U)} \big) - \P\big( \inf_{u\le \eta} Z_u > 0\big) \\
&\ge \P\big( \eta C^{-1} < \tau(U) \big) - \P\big( \inf_{u\le \eta} Z_u > 0\big) \\
&= \P\big( \sup_{t\le \eta C^{-1}} \|X_t - {\overline x}\| <\rho \big) - \P\big( \inf_{u\le \eta} Z_u > 0\big) \\
&\ge \P\big( \sup_{t\le \eta C^{-1}} \|X_t - X_0\| <\rho/2 \big) - \P\big( \inf_{u\le \eta} Z_u > 0\big),
\end{aligned}
\end{equation}
where we recall that $\rho$ is the radius of the open ball $U$, and where the last inequality follows from the triangle inequality provided $\|X_0-{\overline x}\|\le\rho/2$. By Lemma~\ref{L:exitprob} we can choose $\eta>0$ independently of $X_0$ so that $\P( \sup_{t\le \eta C^{-1}} \|X_t - X_0\| <\rho/2 )>1/2$. Then, by Lemma~\ref{L:Bessel}, we have $\P( \inf_{u\le \eta} Z_u > 0)<1/3$ whenever $Z_0=p(X_0)$ is sufficiently close to zero. This happens if $X_0$ is sufficiently close to ${\overline x}$, say within a distance $\rho'>0$. Thus, setting $\varepsilon=\rho'\wedge(\rho/2)$, the condition $\|X_0-{\overline x}\|<\rho'\wedge(\rho/2)$ implies that \eqref{eq:00111} is valid, with the right-hand side strictly positive. The theorem is proved.

\section{Proof of Proposition~\ref{P:quad}}\label{secPquad}
Condition~\ref{G1} is vacuously true, so we prove \ref{G2}. If $d=1$ we have $\{p=0\}=\{-1,1\}$, and it is clear that any univariate polynomial vanishing on this set has $p(x)=1-x^2$ as a factor. Thus \ref{G2} holds. If $d\ge 2$, then $p(x)=1-x^\top Qx$ is irreducible and changes sign, so \ref{G2} follows from Lemma~\ref{L:irred}.

Next, it is straightforward to verify that \eqref{eq:quad:a}--\eqref{eq:quad} imply \ref{T:existence:psd}--\ref{T:existence:q}, so we focus on the converse direction and assume~\ref{T:existence:psd}--\ref{T:existence:q} hold. We first prove that $a(x)$ has the stated form. Write $a(x)=\alpha + L(x) + A(x)$, where $\alpha=a(0)\in\S^d_+$, $L(x)\in\S^d$ is linear in~$x$, and $A(x)\in\S^d$ is homogeneous of degree two in~$x$. Since $a(x)Qx=a(x)\nabla p(x)/2=0$ on $\{p=0\}$, we have for any $x\in\{p=0\}$ and $\epsilon\in\{-1,1\}$,
\[
0 = \epsilon a(\epsilon x) Q x = \epsilon\left( \alpha Qx + A(x)Qx \right) + L(x)Qx.
\]
This implies $L(x)Qx=0$ for all $x\in\{p=0\}$, and thus, by scaling, for all $x\in\R^d$. We now argue that this implies $L=0$. To this end, consider the linear map $T: \Xcal \to \Ycal$ where
\begin{align*}
\Xcal&=\{\text{all linear maps $\R^d\to\S^d$}\}, \\
\Ycal&=\{\text{all second degree homogeneous maps $\R^d\to\R^d$}\},
\end{align*}
and $TK\in\Ycal$ is given by $(TK)(x) = K(x)Qx$. One readily checks that $\dim\Xcal=\dim\Ycal=d^2(d+1)/2$. Thus, if we can show that $T$ is surjective, the rank-nullity theorem implies that $\ker T$ is trivial. But the identity $L(x)Qx\equiv 0$ precisely states that $L\in\ker T$, yielding $L=0$ as desired. To see that $T$ is surjective, note that $\Ycal$ is spanned by elements of the form
\[
(0,\ldots,0,x_ix_j,0,\ldots,0)^\top
\]
with the $k$th component nonzero. Such an element can be realized as $(TK)(x)=K(x)Qx$ as follows: If $i,j,k$ are all distinct, one may take
\[
\begin{pmatrix}
K_{ii}		&	K_{ij}		&K_{ik} \\
K_{ji}		&	K_{jj}		&K_{jk} \\
K_{ki}	&	K_{kj}	&K_{kk}
\end{pmatrix}\!(x)
=
\frac{1}{2} \begin{pmatrix}
0	&-x_k	&x_j \\
-x_k	&0		&x_i \\
x_j	&x_i		&0
\end{pmatrix} \begin{pmatrix}
Q_{ii}&	0		&0 \\
0	&	Q_{jj}	&0 \\
0	&	0		&Q_{kk}
\end{pmatrix},
\]
and all remaining entries of $K(x)$ equal to zero. If $i=k$, one takes $K_{ii}(x)=x_j$ and the remaining entries zero. Similarly if $j=k$. If $i=j\ne k$, one sets
\[
\begin{pmatrix}
K_{ii}		&	K_{ik} \\
K_{ki}	&	K_{kk}
\end{pmatrix}\!(x)
=
\begin{pmatrix}
-x_k	&x_i \\
x_i	&0
\end{pmatrix} \begin{pmatrix}
Q_{ii}&	0 \\
0	&	Q_{kk}
\end{pmatrix},
\]
and the remaining entries zero. This covers all possible cases, and shows that $T$ is surjective. Thus $L=0$ as claimed.

At this point we have shown that $a(x)=\alpha+A(x)$ with $A$ homogeneous of degree two. Next, since $a\,\nabla p=0$ on $\{p=0\}$ there exists a vector $h$ of polynomials such that $a\,\nabla p/2=h\,p$. By counting degree, $h$ is of the form $h(x)=f+Fx$ for some $f\in\R^d$, $F\in\R^{d\times d}$. For any $s>0$ and $x\in\R^d$ such that $sx\in E$,
\[
\alpha Qx + s^2 A(x)Qx = \frac{1}{2s}a(sx)\nabla p(sx) = (1-s^2x^\top Qx)(s^{-1}f + Fx).
\]
By sending $s$ to zero we deduce $f=0$ and $\alpha x=Fx$ for all $x$ in some open set, hence $F=\alpha$. Thus $a(x)Qx=(1-x^\top Qx)\alpha Qx$ for all $x\in E$. Defining $c(x)=a(x) - (1-x^\top Qx)\alpha$, this shows that $c(x)Qx=0$ for all $x\in\R^d$, that $c(0)=0$, and that $c(x)$ has no linear part. In particular, $c$ is homogeneous of degree two. To prove that $c\in\Ccal^Q_+$, it only remains to show that $c(x)$ is positive semidefinite for all $x$. For this we observe that, for any $u\in\R^d$ and any $x\in\{p=0\}$,
\[
u^\top c(x) u = u^\top a(x) u \ge 0.
\]
In view of the homogeneity property, positive semidefiniteness follows for any~$x$. Thus $c\in\Ccal^Q_+$ and hence that $a(x)$ has the stated form. Furthermore, the drift vector is always of the form $b(x)=\beta+Bx$, and a brief calculation using the expressions for $a(x)$ and $b(x)$ shows that the condition $\Gcal p> 0$ on $\{p=0\}$ is equivalent to~\eqref{eq:quad}.

\section{Proof of Proposition~\ref{P:01mn}}\label{secP01mn}
Condition \ref{G1} is vacuously true, and it is not hard to check that \ref{G2} holds.

Next, it is straightforward to verify that {\ref{P:01mn:a}} and {\ref{P:01mn:b}} imply \ref{T:existence:psd}--\ref{T:existence:q}, so we focus on the converse direction and assume~\ref{T:existence:psd}--\ref{T:existence:q} hold.

We first deduce {\ref{P:01mn:a}} from the condition $a\,\nabla p=0$ on $\{p=0\}$ for all $p\in\Pcal$ together with the positive semidefinite requirement of $a(x)$. Taking $p(x)=x_i$, $i=1,\ldots,d$, we obtain $a(x)\nabla p(x) = a(x) e_i = 0$ on $\{x_i=0\}$. Hence the $i$:th column of $a(x)$ is a polynomial multiple of $x_i$. Similarly, with $p=1-x_i$, $i\in I$, it follows that $a(x)e_i$ is a polynomial multiple of $1-x_i$ for $i\in I$. Hence, by symmetry of $a$, we get
\[
\gamma_{ji}x_i(1-x_i) = a_{ji}(x) = a_{ij}(x) = h_{ij}(x)x_j\qquad (i\in I,\ j\in I\cup J)
\]
for some constants $\gamma_{ij}$ and polynomials $h_{ij}\in\Pol_1(E)$ (using also that $\deg a_{ij}\le 2$). For $i\ne j$ this is possible only if $a_{ij}(x)=0$, and for $i=j\in I$ it yields $a_{ii}(x)=\gamma_ix_i(1-x_i)$ as desired. In order to maintain positive semidefiniteness, we necessarily have $\gamma_i\ge 0$.

Now consider $i,j\in J$. By the above, we have $a_{ij}(x)=h_{ij}(x)x_j$ for some $h_{ij}\in\Pol_1(E)$. Similarly as before, symmetry of $a(x)$ yields
\[
h_{ij}(x)x_j = a_{ij}(x) = a_{ji}(x) = h_{ji}(x)x_i,
\]
so that for $i\ne j$, $h_{ij}$ has $x_i$ as a factor. It follows that $a_{ij}(x)=\alpha_{ij}x_ix_j$ for some $\alpha_{ij}\in\R$. If $i=j$, we get $a_{jj}(x)=\alpha_{jj}x_j^2+x_j(\phi_j+\psi_{(j)}^\top x_I + \pi_{(j)}^\top x_J)$ for some $\alpha_{jj}\in\R$, $\phi_j\in\R$, $\psi_{(j)}\in\R^m$, $\pi_{(j)}\in\R^n$ with $\pi_{(j),j}=0$. Positive semidefiniteness requires $a_{jj}(x)\ge 0$ for all $x\in E$. This directly yields $\pi_{(j)}\in\R^n_+$. Further, by setting $x_i=0$ for $i\in J\setminus\{j\}$ and making $x_j>0$ sufficiently small, we see that $\phi_j+\psi_{(j)}^\top x_I\ge0$ is required for all $x_I\in[0,1]^m$, which forces $\phi_j\ge (\psi_{(j)}^-)^\top\oo$. Finally, let $\alpha\in\S^n$ be the matrix with elements $\alpha_{ij}$ for $i,j\in J$, let $\Psi\in\R^{m\times n}$ have columns $\psi_{(j)}$, and $\Pi\in\R^{n\times n}$ have columns $\pi_{(j)}$. We then have
\begin{align*}
s^{-2}\,a_{JJ}(x_I,s x_J) &= \Diag(x_J)\alpha\Diag(x_J) \\
&\qquad  + \Diag(x_J)\Diag(s^{-1}(\phi+\Psi^\top x_I) + \Pi^\top x_J),
\end{align*}
so by sending $s$ to infinity we see that $\alpha + \Diag(\Pi^\top x_J)\Diag(x_J)^{-1}$ must lie in $\S^n_+$ for all $x_J\in\R^n_{++}$. This proves~{\ref{P:01mn:a}}.

For~{\ref{P:01mn:b}}, note that $\Gcal p(x) = b_i(x)$ for $p(x)=x_i$, and $\Gcal p(x)=-b_i(x)$ for $p(x)=1-x_i$. In particular, if $i\in I$, then $b_i(x)$ cannot depend on $x_J$. This establishes~\eqref{eq:P_01mn:b}. Next, for $i\in I$, we have $\beta_i+B_{iI}x_I> 0$ for all $x_I\in[0,1]^m$ with $x_i=0$, and this yields $\beta_i - (B^-_{i,I\setminus\{i\}})\oo > 0$. Similarly, $\beta_i+B_{iI}x_I<0$ for all $x_I\in[0,1]^m$ with $x_i=1$, so that $\beta_i + (B^+_{i,I\setminus\{i\}})\oo + B_{ii}< 0$. For $j\in J$, we may set $x_J=0$ to see that $\beta_J+B_{JI}x_I\in\R^n_{++}$ for all $x_I\in[0,1]^m$. Hence $\beta_j> (B^-_{jI})\oo$ for all $j\in J$. Moreover, fixing $j\in J$, setting $x_j=0$, and letting $x_i\to\infty$ for $i\ne j$ forces $B_{ji}>0$. The proof of~{\ref{P:01mn:b}} is complete.

\section{Proof of Proposition~\ref{P:simplex}}\label{secPsimplex}
Since $\Qcal$ consists of the single polynomial $q(x)=1-\oo^\top x$ it is clear that~\ref{G1} holds. To prove~\ref{G2} it suffices by Lemma~\ref{L:prime} to prove for each~$i$ that the ideal $(x_i, 1-\oo^\top x)$ is prime and has dimension $d-2$. But an affine change of coordinates shows that this is equivalent to same statement for $(x_1,x_2)$, which is well-known to be true.

Next, the only non-trivial aspect of verifying that {\ref{P:01mn:a}} and {\ref{P:01mn:b}} imply \ref{T:existence:psd}--\ref{T:existence:q} is to check that $a(x)$ is positive semidefinite for each $x\in E$. To do this, fix any $x\in E$ and let $\Lambda$ denote the diagonal matrix with $a_{ii}(x)$, $i=1,\ldots,d$ on the diagonal. Then for each $s\in[0,1)$, the matrix $A(s)=(1-s)(\Lambda+\Id)+sa(x)$ is strictly diagonally dominant\footnote{A matrix $A$ is called {\em strictly diagonally dominant} if $|A_{ii}|>\sum_{j\ne i}|A_{ij}|$ for all $i$; see \citet[Definition~6.1.9]{Horn/Johnson:1985}.} with positive diagonal elements. Hence by \citet[Theorem~6.1.10]{Horn/Johnson:1985}, it is positive definite. But since $\S^d_+$ is closed and since $\lim_{s\to1}A(s)=a(x)$, we get $a(x)\in\S^d_+$.

We now focus on the converse direction and assume~\ref{T:existence:psd}--\ref{T:existence:q} hold. We first prove~{\ref{P:simplex:a}}. Since the ideal $(x_i,1-\oo^\top x)$ satisfies \ref{G2} for each $i$, the condition $a(x)e_i=0$ on $M\cap\{x_i=0\}$ implies that
\begin{equation} \label{eq:simplex001}
a_{ji}(x) = x_i h_{ji}(x) + (1-\oo^\top x) g_{ji}(x)
\end{equation}
for some polynomials $h_{ji}$ and $g_{ji}$ in $\Pol_1(\R^d)$. Suppose $j\ne i$. By symmetry of $a(x)$, we get
\[
x_jh_{ij}(x) = x_ih_{ji}(x) + (1-\oo^\top x) (g_{ji}(x) - g_{ij}(x)).
\]
Thus $h_{ij}=0$ on $M\cap\{x_i=0\}\cap\{x_j\ne0\}$, and, by continuity, on $M\cap\{x_i=0\}$. Another application of \ref{G2} and counting degrees gives $h_{ij}(x)=-\alpha_{ij}x_i+(1-\oo^\top x)\gamma_{ij}$ for some constants $\alpha_{ij}$ and $\gamma_{ij}$. This proves $a_{ij}(x)=-\alpha_{ij}x_ix_j$ on $E$ for $i\ne j$, as claimed. For $i=j$, note that \eqref{eq:simplex001} can be written
\[
a_{ii}(x) = -\alpha_{ii}x_i^2 + x_i(\phi_i + \psi_{(i)}^\top x) + (1-\oo^\top x) g_{ii}(x)
\]
for some constants $\alpha_{ij}$, $\phi_i$ and vectors $\psi_{(i)}\in\R^d$ with $\psi_{(i),i}=0$. We need to identify $\phi_i$ and $\psi_{(i)}$. To this end, note that the condition $a(x)\oo=0$ on $\{1-\oo^\top x=0\}$ yields $a(x)\oo=(1-\oo^\top x)f(x)$ for all $x\in\R^d$, where $f$ is some vector of polynomials $f_i\in\Pol_1(\R^d)$. Writing the $i$:th component of $a(x)\oo$ in two ways then yields
\begin{equation} \label{eq:simplex1}
\begin{aligned}
x_i\left( -\sum_{j=1}^d \alpha_{ij}x_j + \phi_i + \psi_{(i)}^\top x\right) &= (1 - \oo^\top x)(f_i(x) - g_{ii}(x)) \\
&= (1 - \oo^\top x)(\eta_i + (\Eta x)_i)
\end{aligned}
\end{equation}
for all $x\in\R^d$ and some $\eta\in\R^d$, $\Eta\in\R^{d\times d}$. Replacing $x$ by $sx$, dividing by $s$, and sending $s$ to zero gives $x_i\phi_i = \lim_{s\to0} s^{-1}\eta_i + (\Eta x)_i$, which forces $\eta_i=0$, $\Eta_{ij}=0$ for $j\ne i$, and $\Eta_{ii}=\phi_i$. Substituting into~\eqref{eq:simplex1} and rearranging yields
\begin{equation} \label{eq:simplex2}
x_i\left(- \sum_{j=1}^d \alpha_{ij}x_j + \psi_{(i)}^\top x + \phi_i \oo^\top x\right) = 0
\end{equation}
for all $x\in\R^d$. The coefficient in front of $x_i^2$ on the left-hand side is $-\alpha_{ii}+\phi_i$ (recall that $\psi_{(i),i}=0$), which therefore is zero. That is, $\phi_i=\alpha_{ii}$. With this in mind, \eqref{eq:simplex2} becomes $x_i \sum_{j\ne i} (-\alpha_{ij}+\psi_{(i),j}+\alpha_{ii})x_j = 0$ for all $x\in\R^d$, which implies $\psi_{(i),j}=\alpha_{ij}-\alpha_{ii}$. At this point we have proved
\begin{align*}
a_{ii}(x) &= -\alpha_{ii}x_i^2 + x_i\left(\alpha_{ii} + \sum_{j\ne i}(\alpha_{ij}-\alpha_{ii})x_j\right) \\
&= \alpha_{ii}x_i(1-\oo^\top x) + \sum_{j\ne i}\alpha_{ij}x_ix_j
\end{align*}
on $E$, which yields the stated form of $a_{ii}(x)$. It remains to show that $\alpha_{ij}\ge 0$ for all $i\ne j$. To see this, suppose for contradiction that $\alpha_{ik}<0$ for some $(i,k)$. Pick $s\in(0,1)$ and set $x_k=s$, $x_j=(1-s)/(d-1)$ for $j\ne k$. Then
\[
a_{ii}(x) = x_i \sum_{j\ne i}\alpha_{ij}x_j = x_i\left(\alpha_{ik}s + \frac{1-s}{d-1}\sum_{j\ne i,k}\alpha_{ij}\right).
\]
For $s$ sufficiently close to $1$ the right-hand side becomes negative, which contradicts positive semidefiniteness of $a$ on $E$. This proves~{\ref{P:simplex:a}}.

For {\ref{P:simplex:b}}, first note that we always have $b(x)=\beta+Bx$ for some $\beta\in\R^d$ and $B\in\R^{d\times d}$. The condition $\Gcal q=0$ on $M$ for $q(x)=1-\oo^\top x$ yields $\beta^\top\oo + x^\top B^\top\oo = 0$ on $M$. Hence by Lemma~\ref{L:irred}, $\beta^\top\oo + x^\top B^\top\oo=\kappa(1-\oo^\top x)$ for all $x\in\R^d$ and some constant $\kappa$. This yields $\beta^\top\oo=\kappa$ and then $B^\top\oo=-\kappa\oo=-(\beta^\top\oo)\oo$. Next, the condition $\Gcal p_i \ge 0$ on $M\cap\{p_i=0\}$ for $p_i(x)=x_i$ can be written
\[
\min \left\{ \beta_i + \textstyle{\sum_{j=1}^d} B_{ji}x_j\ :\ x\in\R^d_+,\ \oo^\top x = \oo,\ x_i=0\right\} \ge 0,
\]
which in turn is equivalent to
\[
\min \left\{ \beta_i + \textstyle{\sum_{j\ne i}} B_{ji}x_j\ :\ x\in\R^d_+,\ \textstyle{\sum_{j\ne i}} x_j=1\right\} \ge 0.
\]
The feasible region of this optimization problem is the convex hull of $\{e_j:j\ne i\}$, and the linear objective function achieves its minimum at one of the extreme points. Thus we obtain $\beta_i+B_{ji} \ge 0$ for all $j\ne i$ and all $i$, as required.

\section{Some notions from algebraic geometry} \label{A:alg}

In this appendix we briefly review some well-known concepts and results from algebra and algebraic geometry. The reader is referred to \cite{Dummit/Foote:2004} and \cite{Bochnak/Coste/Roy:1998} for more details.

An {\em ideal $I$ of $\Pol(\R^d)$} is a subset of $\Pol(\R^d)$ closed under addition such that $f\in I$ and $g\in\Pol(\R^d)$ implies $fg\in I$. Given a finite family $\Rcal=\{r_1,\ldots,r_m\}$ of polynomials, the {\em ideal generated by $\Rcal$}, denoted by $(\Rcal)$ or $(r_1,\ldots,r_m)$, is the ideal consisting of all polynomials of the form $f_1 r_1+\cdots+f_mr_m$, with $f_i\in\Pol(\R^d)$. Given any set of polynomials $S$, its {\em zero set} is the set $\Vcal(S)=\{x\in\R^d:f(x)=0\text{ for all }f\in S\}$. The zero set of the family $\Rcal$ coincides with the zero set of the ideal $I=(\Rcal)$, that is, $\Vcal( \Rcal)=\Vcal(I)$. For example, the set $M$ in~\eqref{eq:M} is the zero set of the ideal~$(\Qcal)$. Given a set $V\subseteq\R^d$, the {\em ideal generated by~$V$}, denoted by $\Ical(V)$, is the set of all polynomials that vanish on $V$. It follows from the definition that $S\subseteq \Ical(\Vcal(S))$ for any set of polynomials $S$.  A basic problem in algebraic geometry is to establish when an ideal $I$ is equal to the ideal generated by the zero set of $I$,
\begin{equation}\label{IIVI}
I = \Ical(\Vcal(I)).
\end{equation}
If the ideal $I=(\Rcal)$ satisfies \eqref{IIVI} then that means that any polynomial $f$ that vanishes on the zero set $\Vcal(I)$ has a representation $f=f_1r_1+\cdots+f_mr_m$ for some polynomials $f_1,\ldots,f_m$.

An ideal $I$ of $\Pol(\R^d)$ is said to be {\em prime} if it is not all of $\Pol(\R^d)$ and if the conditions $f,g\in \Pol(\R^d)$ and $fg\in I$ imply $f\in I$ or $g\in I$. The {\em dimension} of an ideal $I$ of $\Pol(\R^d)$ is the dimension of the quotient ring $\Pol(\R^d)/I$; for a definition of the latter, see \citet[Section~16.1]{Dummit/Foote:2004}.

\bibliographystyle{plainnat}


\bibliography{bibl}

\end{document}